\documentclass[11pt,reqno]{amsart}
\usepackage{hyperref}\hypersetup{colorlinks=true, citecolor=blue}
\usepackage{graphicx}
\usepackage{esint,amsfonts,amsmath,amssymb,epsfig,
mathrsfs,bm,accents,yfonts,mathtools,stmaryrd}

\numberwithin{equation}{section}

\theoremstyle{plain}
\newtheorem{maintheorem}[subsubsection]{Theorem}
\newtheorem{theorem}[subsubsection]{Theorem}
\newtheorem{lemma}[subsubsection]{Lemma}

\newtheorem{proposition}[subsubsection]{Proposition}
\newtheorem{corollary}[subsubsection]{Corollary}

\theoremstyle{definition}
\newtheorem{definition}[subsubsection]{Definition}

\theoremstyle{remark}
\newtheorem{remark}[subsubsection]{Remark}
\newtheorem*{remark*}{Remark}

\newcommand{\R}{\mathbb{R}}

\newcommand{\be}{{\mathbf{e}}}

\newcommand{\eps}{{\varepsilon}}
\newcommand{\mint}{-\!\!\!\!\!\!\int}

\newcommand{\Lip}{{\text {Lip}}}

\newcommand{\dist}{{\text {dist}}}
\newcommand{\im}{{\text {Im}}}

\newcommand{\dv}{{\text {div}}}

\newcommand\weak{{\rightharpoonup}\,}

\renewcommand\d{{\rm d}\,}

\newcommand\supp{{\rm spt}\,}
\newcommand\res{\mathop{\hbox{\vrule height 7pt width .5pt depth 0pt
\vrule height .5pt width 6pt depth 0pt}}\nolimits}
\newcommand\Id{{\rm Id}\,}

\newcommand{\mass}{{\mathbf{M}}}

\newcommand{\bC}{{\bar{C}}}

\newcommand{\cB}{{\mathcal{B}}}
\newcommand{\cG}{{\mathcal{G}}}

\newcommand{\cL}{{\mathcal{L}}}
\newcommand{\cH}{{\mathcal{H}}}

\newcommand{\cM}{{\mathcal{M}}}

\newcommand{\cC}{{\mathcal{C}}}

\newcommand{\cU}{{\mathcal{U}}}

\newcommand{\cR}{{\mathcal{R}}}

\newcommand{\Pe}{{\mathbf{S}}}
\newcommand\sD{{\mathscr D}}

\newcommand\Z{{\mathbb Z}}

\newcommand\N{{\mathbb N}}

\newcommand\C{{\mathbb C}}

\newcommand{\Iqs}{{\mathcal{A}}_Q(\R^{n})}
\newcommand{\Iq}{{\mathcal{A}}_Q}
\def\a#1{\left\llbracket{#1}\right\rrbracket}
\newcommand{\abs}[1]{\left|#1\right|}

\newcommand{\D}{\textup{Dir}}
\newcommand{\de}{\partial}

\newcommand{\etaa}{{\bm{\eta}}}
\newcommand{\ph}{\varphi}

\newcommand{\graph}{\textup{graph}}
\newcommand{\loc}{\textup{loc}}
\newcommand\bT{\mathbf{T}}
\newcommand{\bG}{\mathbf{G}}
\newcommand{\gr}{{\rm Gr}}
\newcommand{\p}{{\bf p}}
\newcommand{\bE}{{\mathbf{E}}}

\newcommand{\bh}{{\mathbf{h}}}
\newcommand\rD{{\rm D}}
\newcommand\B{{\mathbf{B}}}
\newcommand{\sW}{{\mathscr{W}}}
\newcommand{\sC}{{\mathscr{C}}}
\newcommand\sS{{\mathscr S}}
\newcommand\sP{{\mathscr P}}

\newcommand{\bGam}{{\bm \Gamma}}

\newcommand{\bU}{{\mathbf{U}}}

\newcommand{\phii}{{\bm{\varphi}}}
\newcommand{\Phii}{{\bm{\Phi}}}
\newcommand{\cV}{{\mathcal{V}}}

\newcommand{\cK}{{\mathcal{K}}}

\newcommand\Sing{\textup{Sing}}
\newcommand\Reg{\textup{Reg}}

\newcommand{\bD}{{\mathbf{D}}}
\newcommand{\bH}{{\mathbf{H}}}
\newcommand{\bI}{{\mathbf{I}}}
\newcommand{\bF}{{\mathbf{F}}}
\newcommand{\bOmega}{{\mathbf{\Omega}}}
\newcommand{\bSigma}{{\mathbf{\Sigma}}}

\author{Emanuele Spadaro}
\title[Higher codimension integral currents]{Regularity of higher codimension area minimizing integral currents}
\email{spadaro@mis.mpg.de}

\begin{document}

\begin{abstract}
This lecture notes are an expanded and revised version of the course 
\textit{Regularity of higher codimension area minimizing integral currents}
that I taught at the \textit{ERC-School on Geometric Measure Theory and Real Analysis},
held in Pisa, September 30th - October 30th 2013.

The lectures aim to explain
the main steps of a new proof of the partial
regularity of area minimizing integer rectifiable currents in higher codimension, due originally to F. Almgren,
which is contained in a series of papers in collaboration with C. De Lellis (University of Z\"urich).
\end{abstract}

\maketitle

\tableofcontents

%
%
\section{Introduction}

The subject of this course is the study of the regularity of \textit{minimal surfaces},
considered in the sense of \textit{area minimizing integer rectifiable currents}.
This is a very classical topic and stems from many diverse questions and applications.
Among the most known there is perhaps the so called \textit{Plateau problem},
consisting in finding the submanifolds of least possible volume among all those
submanifolds with a fixed boundary.

\medskip

\textbf{Plateau problem.} Let $M$ be a $(m+n)$-dimensional Riemannian
manifold and $\Gamma\subset M$ a compact $(m-1)$-dimensional oriented submanifold.
Find an $m$-dimensional oriented submanifold $\Sigma$ with boundary $\Gamma$ such that
\[
\textup{vol}_m (\Sigma) \leq \textup{vol}_m (\Sigma'),
\]
for all oriented submanifolds $\Sigma' \subset M$ such that $\de \Sigma' = \Gamma$.

\medskip

It is a well-known fact that the solution of the Plateau problem does not 
always exist.
For example, consider $M = \R^4$, $n=m=2$ and $\Gamma$ the smooth Jordan curve parametrized in the following way:
\[
\Gamma = \big\{ (\zeta^2, \zeta^3) : \zeta \in \C, \; |\zeta|=1 \big\} \subset \C^2 \simeq \R^4,
\]
where we use the usual identification between $\C^2$ and $\R^4$, and we choose the orientation
of $\Gamma$ induced by the anti-clockwise orientation of the unit circle $|\zeta|=1$ in $\C$.
It can be shown (and we will come back to this point in the next sections) that
there exist no smooth solutions to the Plateau problem for such fixed boundary,
and the \textit{(singular) immersed} $2$-dimensional disk
\[
S = \big\{(z,w) \;:\; z^3 = w^2, \; |z|\leq 1 \big\} \subset \C^2 \simeq \R^4,
\]
oriented in such a way that $\de S = \Gamma$, satisfies
\[
\cH^2(S) < \cH^2(\Sigma),
\]
for all smooth, oriented $2$-dimensional submanifolds $\Sigma \subset \R^4$ with
$\de \Sigma = \Gamma$. Here and in the following we denote by $\cH^k$ the $k$-dimensional
Hausdorff measure, which for $k \in \N$ corresponds to the ordinary $k$-volume
on smooth $k$-dimensional submanifolds.

\medskip

This fact motivates the introduction of \textit{weak solutions} to the Plateau problem,
and the main questions about their existence and regularity.

\subsection{Integer rectifiable currents}
One of the most successful theories of generalized submanifolds
is the one by H. Federer and W. Fleming in \cite{FF} on
integer rectifiable currents (see also \cite{DG54,DG55} for the special case of codimension one generalized submanifolds).
From now on, in order to keep the technicalities to a minimum level, we assume that our ambient Riemannian manifold $M$ is Euclidean.

\begin{definition}[Integer rectifiable currents]\label{d:ir}
An integer rectifiable current $T$ of dimension $m$ in $\R^{m+n}$ is a
triple $T=(R, \tau, \theta)$ such that:
\begin{itemize}
\item[(i)] $R$ is a \textit{rectifiable set}, i.e.~$R = \bigcup_{i \in \N} C_i$
with $\cH^m(R_0) = 0$ and $C_i \subset M_i$ for every $i \in \N\setminus \{0\}$,
where $M_i$ are $m$-dimensional oriented $C^11$ submanifolds of $\R^{m+n}$;

\item[(ii)] $\tau:R \to \Lambda_m$ is a measurable map,
called \textit{orientation},
taking values in the space of $m$-vectors
such that, for $\cH^m$-a.e.~$x \in C_i$, $\tau(x) = v_1\wedge \cdots \wedge v_m$ with
$\{v_1, \ldots, v_m\}$ an oriented orthonormal basis of $T_x M_i$;

\item[(iii)] $\theta:R \to \Z$ is a measurable function,
called \textit{multiplicity},
which is integrable with respect to $\cH^m$.
\end{itemize}
An integer rectifiable current $T= (R, \tau, \theta)$ induces a continuous linear functional (with respect to
the natural Fr\'echet topology) on smooth, compactly supported
$m$-dimensional differential forms $\omega$, denoted by $\sD^m$,
acting as follows
\[
T(\omega) = \int_R \theta\,\langle \omega, \tau \rangle \,d\cH^m. 
\]
\end{definition}

\begin{remark}
The continuous linear functionals defined in the Fr\'echet space $\sD^m$
are called \textit{$m$-dimensional currents}.
\end{remark}

\begin{remark}
Note that the submanifold $M_i$ in Definition~\ref{d:ir} are only $C^1$ regular. This
restriction is not redundant, but it is connected to several aspects of the theory
of rectifiable sets.
\end{remark}

For an integer rectifiable current $T$, one can define the analog of the boundary and the volume
for smooth submanifolds.

\begin{definition}[Boundary and mass]\label{d:bdy and mass}
Let $T=(R, \tau, \theta)$ be an integer rectifiable current in $\R^{m+n}$ of dimension $m$.
The \textit{boundary} of $T$ is defined as the $(m-1)$-dimensional current acting as follows
\[
\de T (\omega) := T(d\omega) \quad \forall\; \omega \in \sD^{m-1}.
\]
The \textit{mass} of $T$ is defined as the quantity
\[
\mass(T) := \int_R |\theta| \,d\cH^m.
\]
\end{definition}

Note that, in the case $T= (\Sigma, \tau_\Sigma, 1)$ is the current
induced by an oriented submanifold $\Sigma$ with boundary $\de \Sigma$,
with $\tau_\Sigma$ a continuous orienting vector for $\Sigma$ and similarly
$\tau_{\de \Sigma}$ for its boundary,
then by Stoke's Theorem $\de T = (\de \Sigma, \tau_{\de \Sigma}, 1)$ and
$\mass(T) = \textup{vol}_m(\Sigma)$.

Finally we recall that the space of currents is usually endowed with the weak* topology
(often called in this context \textit{weak} topology).

\begin{definition}[Weak topology]
We say that a sequence of currents $(T_l)_{l\in \N}$ weakly converges to some current
$T$, and we write $T_l \weak T$, if
\[
T_l(\omega) \to T(\omega) \quad \forall\; \omega \in \sD^m.
\]
\end{definition}

\bigskip

The Plateau problem has now a straightforward generalization in this context of integer rectifiable currents.

\medskip

\textbf{Generalized Plateau problem.} Let $\Gamma$ be a compactly supported
$(m-1)$-dimensional integer rectifiable current in $\R^{m+n}$ with $\de\Gamma=0$.
Find an $m$-dimensional integer rectifiable
current $T$ such that $\de T= \Gamma$ and
\[
\mass(T) \leq \mass(S),
\]
for every $S$ integer rectifiable with $\de S = \Gamma$.

\medskip

The success of the theory of integer rectifiable currents is linked ultimately
to the possibility to solve the generalized Plateau problem, 
due to the closure
theorem by H. Federer and W. Fleming proven in their pioneering paper \cite{FF}.

\begin{theorem}[Federer and Fleming \cite{FF}]
Let $(T_l)_{l\in \N}$ be a sequence of $m$-dimensional integer rectifiable currents
in $\R^{m+n}$ with
\[
\sup_{l\in \N} \big( \mass(T_l) + \mass(\de T_l) \big) < + \infty,
\]
and assume that $T_l \weak T$.
Then, $T$ is an integer rectifiable current.
\end{theorem}

It is then natural to ask about the regularity properties of the solutions to
the generalized Plateau problem, called in the sequel \textit{area minimizing}
integer rectifiable currents.

\subsection{Partial regularity in higher codimension}
The regularity theory for area minimizing integer rectifiable currents
depends very much on the dimension of the current and its \textit{codimension} in
the ambient space
(i.e., using the same letters as above, if $T$ is an $m$-dimensional current in $\R^{m+n}$, the codimension is $n$).

In this course we are interested in the general case of currents with higher codimensions $n > 1$.
The case $n=1$ is usually treated separately, because 
different techniques can be used and more refined results can be proven 
(see \cite{DG, Fleming, Sim83, Sim95, Simons, Reif} for the
interior regularity and \cite{All2, HS} for the boundary regularity).
In higher codimension the most general result is due to F.~Almgren \cite{Alm}
and concerns the interior partial regularity up to a (relatively) closed set of dimension
at most $m-2$.

\begin{maintheorem}[Almgren \cite{Alm}]\label{t:main}
Let $T$ be an $m$-dimensional area minimizing integer rectifiable current in $\R^{m+n}$.
Then, there exists a closed set $\Sing(T)$ of Hausdorff dimension at most $m-2$
such that in $\R^{m+n} \setminus (\supp(\de T) \cup \Sing(T))$
the current $T$ is induced by the integration over a smooth oriented submanifold of $\R^{m+n}$.
\end{maintheorem}

In the next pages I will give an overview of the new proof of Theorem~\ref{t:main}
given in collaboration with C.~De Lellis in a series of papers
\cite{DS1, DS2, DS3, DS4, DS5}.
Although our proof is considerably simpler than the original one,
it remains quite involved:
this text is, therefore, meant as
a survey of the techniques and the various steps
of the proof, and can be
considered an introduction to the reading of the papers \cite{DS3, DS4, DS5}.

\begin{remark}
The interior partial regularity can be proven for integer rectifiable currents
in a Riemannian manifold $M$.
In \cite{Alm} Almgren proves the result for $C^5$ regular ambient
manifolds $M$, while our papers \cite{DS3, DS4, DS5} extend this result to $C^{3,\alpha}$ regular manifolds. 
\end{remark}

\addtocontents{toc}{\protect\setcounter{tocdepth}{1}}

\subsection*{Further notation and terminology}
Given an $m$-dimensional integer rectifiable current $T = (R, \tau, \theta)$, we shall often use the
following standard notation:
\begin{gather*}
\|T\| := |\theta| \, \cH^m \res R,\quad
\vec T := \tau\quad\text{and}\quad
\supp(T) := \supp(\|T\|).
\end{gather*}
The regular and the singular part of a current are defined as follows.
\begin{multline*}
\Reg (T)  := \big\{x\in \supp (T) \;:\; \supp(T)\cap  B_r (x) \textup{ is induced by a smooth}\\
\textup{  submanifold for some $r>0$}\big\},
\end{multline*}
\begin{equation*}
\Sing (T)  := \supp (T) \setminus \big(\supp (\partial T) \cup \Reg (T)\big).
\end{equation*}

\subsection*{Acknowledgements}
I am very grateful to A.~Marchese, for reading a first draft of
these lecture notes and suggesting many precious improvements.

%
%
\section{The blowup argument: a glimpse of the proof}
\addtocontents{toc}{\protect\setcounter{tocdepth}{2}}

The main idea of the proof of Theorem~\ref{t:main} is to
detect the singularities of an area minimizing current
by a blowup analysis.
For any $r>0$ and $x\in \R^{m+n}$, let $\iota_{x,r}$ denote the map
\[
\iota_{x,r}: y \mapsto \frac{y-x}{r},
\]
and set $T_{x, r} := (\iota_{x, r})_{\sharp} T$, where $_\sharp$ is the push-forward operator, namely
\[
(\iota_{x, r})_{\sharp} T (\omega) := T (\iota_{x, r}^*  \omega) \quad \forall\; \omega \in \sD^m.
\]
By the classical monotonicity formula
(see, e.g., \cite[Section 5]{All}), for every $r_k\downarrow 0$ and $x\in \supp(T) \setminus \supp(\de T)$,
there exists a subsequence (not relabeled) such that
\[
T_{x, r_k}\weak S,
\] 
where $S$ is a cone without boundary (i.e.~$S_{0,r} = S$ for all $r>0$ and $\partial S=0$)
which is locally area minimizing in $\R^{m+n}$.
Such a cone will be called, as usual, {\em a tangent cone to $T$ at $x$}.

\medskip

The idea of the blowup analysis dates back to De Giorgi's pioneering paper \cite{DG} and has
been used in the context of codimension one currents to recognize singular
points and regular points, because in this case the tangent cones to singular and regular points are in fact different.

\subsection{Flat tangent cones do not imply regularity}
This is not the case for higher codimension currents.
In order to illustrate this point, let us consider 
the current $T_{\mathscr{V}}$
induced by the complex curve considered  above:
\[
\mathscr{V} = \big\{(z,w) \;:\; z^3 = w^2, \; |z|\leq 1 \big\} \subset \C^2 \simeq \R^4.
\]
It is simple to show that $T_{\mathscr{V}}$
is an area minimizing integer rectifiable current (cp.~\cite[5.4.19]{Fed}),
which is singular in the origin. 
Nevertheless, the unique tangent cone to $T_{\mathscr{V}}$ at $0$
is the current $S=(\R^2 \times \{0\}, e_1\wedge e_2, 2)$
which is associated to the integration on the horizontal plane $\R^2 \times \{0\} \simeq \{w=0\}$
with multiplicity two.
The tangent cone is actually regular, although the origin is a singular point!

\subsection{Non-homogeneous blowup}
One of the main ideas by Almgren is then to extend this reasoning to different types of blowups,
by rescaling differently the ``horizontal directions'', namely those of a flat tangent cone at the point,
and the ``vertical'' ones, which are the orthogonal complement to the former.
In this way, in place of preserving the geometric properties of the rectifiable current $T$, one is
led to preserve the \textit{energy} of the associated \textit{multiple valued function}.

In order to explain this point, let us consider again the current $T_{\mathscr{V}}$.
The support of such current, namely the complex curve ${\mathscr{V}}$, can be viewed as the graph
of a function which associates to any $z \in \C$ with $|z|\leq 1$ \textit{two} points in
the $w$-plane:
\begin{equation}\label{e:multi}
z \mapsto \{w_1(z), w_2(z) \} \quad \text{with $w_i(z)^2=z^3$ for $i=1, 2$.}
\end{equation}
Then the right rescaling according to Almgren is the one producing
in the limit
a multiple valued \textit{harmonic function} preserving the \textit{Dirichlet energy}
(for the definitions see the next sections).
In the case of $\mathscr{V}$, the correct rescaling is the one fixing $\mathscr{V}$.
For every $\lambda>0$, we consider
$\Phi_\lambda : \C^2 \to \C^2$ given by
\[
\Phi_\lambda( z,w) = (\lambda^2\,z,  \lambda^3\,w),
\]
and note that $(\Phi_\lambda)_\sharp T_\mathscr{V} = T_\mathscr{V}$ for every $\lambda>0$.
Indeed, in the case of $\mathscr{V}$ the functions $w_1$ and $w_2$, being the two determinations
of the square root of $z^3$, are already harmonic functions
(at least away from the origin).

\subsection{Multiple valued functions}

Following these arguments, we have then to face the problem of defining 
harmonic multiple valued functions, and to study their singularities.
Abstracting from the above example, we consider the multiple valued functions
from a domain in $\R^m$ which take a fixed number $Q \in \N \setminus\{0\}$
of values in $\R^n$.
This functions will be called in the sequel \textit{$Q$-valued functions}.

The definition of harmonic $Q$-valued functions is a simple issue
around any ``regular point'' $x_0 \in \R^m$, for 
it is enough to consider just the superposition of
classical harmonic functions (possibly with a constant integer multiplicity), i.e.
\begin{equation}\label{e:reg}
\R^m \supset B_r(x_0)\ni x \mapsto \{u_1(x), \ldots, u_Q(x)\} \in (\R^n)^Q,
\end{equation}
with $u_i$ harmonic and either $u_i = u_j$ or $u_i(x) \neq u_i(x)$ for every $x \in B_r(x_0)$.

\medskip

The issue becomes much more subtle around the singular points. As it is clear from
the example \eqref{e:multi}, in a neighborhood of the origin there is no representation of
the map $z \mapsto \{w_1(z), w_2(z) \}$ as in \eqref{e:reg}.
In this case the two values $w_1(z)$ and $w_2(z)$ cannot be ordered in a consistent way
(due to the \textit{branch point} at $0$),
and hence cannot be distinguished one from the other.
We are then led to consider a multiple valued function as a map taking $Q$ values
in the quotient space $(\R^n)^Q/\sim$
induced by the symmetric group $\Pe_Q$ of permutation of $Q$ indices:
namely, given points $P_i, S_i\in \R^n$,
\[
(P_1, \ldots, P_Q) \sim (S_1, \ldots, S_Q)
\]
if there exists $\sigma \in \Pe_Q$ such that $P_i = S_{\sigma(i)}$ for every $i=1, \ldots, Q$.

Note that the space $(\R^n)^Q/\sim$ is a \textit{singular metric space} (for a naturally
defined metric, see the next section).
Therefore, harmonic maps with values in $(\R^n)^Q/\sim$ have to be carefully defined,
for instance by using the metric theory of harmonic functions developed in 
\cite{GrSc, Jost, KoSc1} (cp.~also \cite{DS1, LoSp}).

\begin{remark}
Note that the integer rectifiable current induced by the graph of a $Q$-valued function
(under suitable hypotheses, cp.~\cite[Proposition~1.4]{DS2})
belongs to a subclass of currents, sometimes called
``positively oriented'', i.e.~such that the tangent planes make at almost
every point a positive angle with a fixed plane.
Nevertheless, as it will become clear along the proof, it is enough to consider
this subclass as model currents in order to conclude Theorem~\ref{t:main}.
\end{remark}

\subsection{The need of centering}
A major geometric and analytic problem has to be addressed in the blowup
procedure sketched above.
In order to make it apparent, let us discuss another example.
Consider the complex curve $\mathscr{W}$ given by
\[
\mathscr{W} = \big\{(z,w) \;:\; (w-z^2)^2 = z^5, \; |z|\leq 1 \big\} \subset \C^2.
\]
As before, $\mathscr{W}$ can be associated to
an area minimizing integer rectifiable current $T_{\mathscr{W}}$ in $\R^4$, which is singular at the origin.
It is easy to prove that
the unique tangent plane to $T_{\mathscr{W}}$ at $0$ is the plane $\{w=0\}$ taken with multiplicity two.
On the other hand, by simple analytical considerations, the only nontrivial inhomogeneous blowup in these
vertical and horizontal coordinates is given by
\[
\Phi_\lambda( z,w) = (\lambda\,z,  \lambda^2\,w),
\]
and $(\Phi_\lambda)_\sharp T_{\mathscr{W}}$ 
converges as $\lambda\to +\infty$ to the current induced by the \textit{smooth} complex curve
$\{w=z^2\}$ taken with multiplicity two.
In other words, the inhomogeneous blowup did not produce in the limit any
singular current and cannot be used to study the singularities of $T_{\mathscr{W}}$.

\medskip

For this reason it is essential to ``renormalize'' $T_{\mathscr{W}}$ by
averaging out its regular first expansion,
on top of which the singular branching behavior happens.
In the case we handle, the regular part of $T_{\mathscr{W}}$ is exactly the smooth
complex curve $\{w=z^2\}$, while the singular branching is due to the determinations
of the square root of $z^5$.
It is then clear why one can look for parametrizations of $\mathscr{W}$
defined in $\{w=z^2\}$, 
so that the singular map to be considered reduces to
\[
z \mapsto \{u_1(z), u_2(z)\} \quad \text{with $u_1(z)^2 = z^5$.}
\]

\medskip

The regular surface $\{w=z^2\}$ is called \textit{center manifold} by Almgren,
because it behaves like (and in this case it is exactly) the average of the
sheets of the current in a suitable system of coordinates.
In general the determination of the center manifold is not straightforward as
in the above example, and actually constitutes the most intricate part of the proof.

\subsection{Excluding an infinite order of contact}
Having taken care of the geometric problem of the averaging,
in order to be able to perform successfully the
inhomogeneous blowup, one has to be sure that the first singular
expansion of the current around its regular part does not occur with
an infinite order of contact, because in that case the blowup would be by
necessity zero.

\medskip

This issue involves one of the most interesting and original ideas of F.~Almgren,
namely a new monotonicity formula for the so called \textit{frequency function}
(which is a suitable ratio between the energy and a zero degree norm of
the function parametrizing the current).
This is in fact the right monotone quantity for the inhomogeneous blowups
introduced before, and it allows to show that the first singular
term in the ``expansion'' of the current does not occur with infinite order of contact
and actually leads to a nontrivial limiting current.

\subsection{The persistence of singularities}
Finally, in order to conclude the proof we need to assure that the singularities
of the current do transfer to singularities of the limiting multiple valued
function, which can be studied with more elementary techniques.
This is in general not true in a pointwise sense,
but it becomes true in a measure theoretic sense as soon as the singular set
is supposed to have positive $\cH^{m-2+\alpha}$ measure, for some $\alpha>0$.

\medskip

The contradiction is then reached in the following way: starting from
an area minimizing current with a big singular set ($\cH^{m-2+\alpha}$ positive
measure), one can perform the analysis outlined before and will end up
with a multiple valued function having a big set of singularities, thus giving
the desired contradiction.

\subsection{Sketch of the proof}\label{ss:sketch}
The rigorous proof of Theorem~\ref{t:main} is actually much more involved and
complicated than the rough outline given in the previous section,
and can be found either in \cite{Alm} or in
the recent series of papers \cite{DS1, DS2, DS3, DS4, DS5}.
In this lecture notes we give some more details of this recent new proof,
and comments on some of the subtleties which were hidden in
the general discussion above.
Since the proof is very lengthly, we start with a description of the strategy.

\bigskip

The proof is done by contradiction.
We will, indeed, always assume the following in the sequel.

\medskip

\textbf{Contradiction assumption:} there exist numbers
$m\geq 2$, $n \geq 1$, $\alpha>0$
and an area minimizing $m$-dimensional integer rectifiable current
$T$ in $\R^{m+n}$ such that
\[
\cH^{m-2+\alpha}(\Sing(T)) > 0.
\]

\medskip

Note that the hypothesis $m\geq 2$ is justified because, for $m=1$
an area minimizing current
is locally the union of finitely many non-intersecting open segments.

\medskip

The aim of the proof is now to show that there exist suitable points
of $\Sing(T)$ where we can perform the blowup analysis outlined in the
previous section.
This process consists of different steps, which we next list
in a way which does not require the introduction of new notation
but needs to be further specified later.

\medskip

\noindent \textbf{(A)} Find a point $x_0 \in \Sing(T)$ and a sequence of radii
$(r_k)_k$ with $r_k \downarrow 0$ such that:
\begin{itemize}
\item[(A$_1$)] the rescaling currents $T_{x_0,r_k} := (\iota_{x_0,r_k})_\sharp T$
converge to a flat tangent cone;
\item[(A$_2$)] $\cH^{m-2+\alpha}(\Sing(T_{x_0,r_k})\cap B_1) > \eta >0$ for some
$\eta>0$ and for every $k \in \N$.
\end{itemize}
Note that both conclusions hold for suitable subsequences, which in principle may
not coincide. What we need to prove is that we can select a point
and a subsequence satisfying both.

\medskip 

\noindent \textbf{(B)} Construction of the center manifold $\cM$ and of a
normal Lipschitz approximation $N:\cM \to \R^{m+n}/\sim$.

This is the most technical part of the proof, and most of the conclusions
of the next steps will intimately depend on this construction.

\medskip 

\noindent \textbf{(C)} The center manifold that one constructs in step (B) can only be
used in general for a finite number of radii $r_k$ of step (A).
The reason is that in general its degree of approximation
of the average of the minimizing currents $T$ is under control
only up to a certain distance from the singular point under consideration.
This leads us to define the sets where the approximation works, called
in the sequel \textit{intervals of flattening}, and to define an entire
\textit{sequence of center manifolds} which will be used in the blowup
analysis.

\medskip 

\noindent \textbf{(D)} Next we will take care of the problem of the infinite order of contact.
This is done in two part. For the first one we derive the \textit{almost
monotonicity formula} for a variant of Almgren's frequency function,
deducing that the order of contact remains finite within each center manifold
of the sequence in (C).

\medskip 

\noindent \textbf{(E)} Then one needs to compare different center manifolds and
to show that the order of contact still remains finite.
This is done by exploiting a deep consequence of the construction in (C)
which we call \textit{splitting before tilting} after the inspiring paper by T.~Rivi{\`e}re \cite{Ri04}.

\medskip 

\noindent \textbf{(F)} With this analysis at hand, we can pass into the limit our
blowup sequence and conclude the convergence to the graph of a harmonic
$Q$-valued function $u$.

\medskip 

\noindent \textbf{(G)} Finally, we discuss the capacitary argument leading to the persistence
of the singularities, to show that the 
function $u$ in (F) needs to have a singular set with positive $\cH^{m-2+\alpha}$ measure,
thus contradicting the partial regularity estimate for such multiple valued
harmonic functions.

\medskip

In the remaining part of this course we give a more detailed description
of the steps above, referring to the original papers
\cite{DS1, DS2, DS3, DS4, DS5} for the complete proofs.

%
%
\section{$Q$-valued functions and rectifiable currents}
Since the final contradiction argument relies on the
regularity theory of multiple valued functions, we start recalling the
main definitions and results concerning them,
and the way they can be used to approximate integer rectifiable currents.
The reference for this part of the theory is \cite{DS1,DS2,DS3,Sp10}.

\subsection{$Q$-valued functions}
We start by giving a metric structure to the
space $(\R^n)^Q/\sim$ of unordered $Q$-tuples of points in $\R^n$,
where $Q \in \N\setminus\{0\}$ is a fixed number.
It is immediate to see that
this space can be identified with the subset of positive measures
of mass $Q$ which are the sum of integer multiplicity Dirac delta:
\[
(\R^n)^Q/\sim \quad \simeq \quad \Iq(\R^n) := \left\{\sum_{i=1}^Q \a{P_i}\;:\; P_i \in \R^n\right\}, 
\]
where $\a{P_i}$ denotes the Dirac delta at $P_i$.
We can then endow $\Iq$ with one of the distances
defined for (probability) measures, for example the
Wasserstein distance of exponent two: for every $T_1=\sum_i\a{P_i}$ and $T_2=\sum_i\a{S_i} \in \Iqs$, we set  
\begin{equation*}
\cG(T_1,T_2):=\min_{\sigma\in\Pe_Q}\sqrt{\sum_{i=1}^Q\abs{P_i-S_{\sigma(i)}}^2},
\end{equation*}
where we recall that $\Pe_Q$ denotes the symmetric group of $Q$ elements.

\medskip

A $Q$-function simply a map $f:\Omega\to \Iqs$, where $\Omega \subset\R^m$ is an
open domain. We can then talk about measurable (with respect to the Borel $\sigma$-algebra of $\Iqs$),
bounded, uniformly-, H\"older- or Lipschitz-continuous $Q$-valued functions.

More importantly, following the pioneering approach to weakly differentiable functions with
values in a metric space by L.~Ambrosio \cite{Amb}, we can also define
the class of Sobolev $Q$-valued functions $W^{1,2}$.

\begin{definition}[Sobolev $Q$-valued functions]\label{d:W1p}
Let $\Omega\subset\R^m$ be a bounded open set.
A measurable function $f:\Omega\to\Iq$ is in the Sobolev class
$W^{1,2}$ if there exist $m$ functions $\varphi_j\in L^2(\Omega)$ for 
$j=1, \ldots, m$, such that 
\begin{itemize}
\item[(i)] $x\mapsto\cG (f(x),T)\in W^{1,2}(\Omega)$ for all $T\in \Iq$;
\item[(ii)] $\abs{\de_j\, \cG (f, T)}\leq\varphi_j$ almost everywhere in $\Omega$
for all $T\in \Iq$ and for all $j\in\{1, \ldots, m\}$,
where $\de_j\cG(f,T)$ denotes the weak partial derivatives
of the functions in (i).
\end{itemize}
\end{definition}

By simple reasonings, one can infer
the existence of minimal functions $|\de_j f|$ fulfilling (ii):
\begin{equation*}
|\de_j f|\leq\varphi_j\; \text{a.e.}\;\text{for any other $\varphi_j$ satisfying (ii),}
\end{equation*}
We set
\begin{equation}\label{e:def|Df|}
|Df|^2 :=\sum_{j=1}^m\abs{\de_jf}^2,
\end{equation}
and define the Dirichlet energy of a $Q$-valued function as
(cp.~also \cite{Jost, KoSc1, LoSp} for alternative definitions)
\[
\textup{Dir}(f) := \int_{\Omega} |Df|^2.
\]
A $Q$-valued function $f$ is said \textit{Dir-minimizing} if 
\begin{gather}\label{e:Dirichlet}
\int_\Omega |Df|^2 \leq \int_{\Omega}|Dg|^2\\
\text{for all $g\in W^{1,2}(\Omega,\Iq)$ with } \cG(f,g)|_{\partial \Omega}=0,\notag
\end{gather}
where the last inequality is meant in the sense of traces.

\medskip

The main result in the theory of $Q$-valued functions is the
following.

\begin{theorem}\label{t:Qvalued}
Let $\Omega\subset \R^m$ be a bounded open domain with Lipschitz boundary,
and let $g \in W^{1,2}(\Omega,\Iqs)$ be fixed.
Then, the following holds.

\begin{itemize}
\item[(i)] There exists a Dir-minimizing function $f$ solving the minimization
problem \eqref{e:Dirichlet}.

\item[(ii)] Every such function $f$ belongs to $C^{0,\kappa}_{\loc} (\Omega,\Iqs)$
for a dimensional constant $\kappa=\kappa(m, Q)>0$.

\item[(iii)] For every such function $f$, $|Df| \in L^p_{\loc}(\Omega)$ for some dimensional
constant $p=p(m,n,Q)>2$.

\item[(iv)] There exists a relatively
closed set $\Sing(u) \subset \Omega$ of Hausdorff dimension at most $m-2$
such that the graph of $u$ outside 
$\Sing(u)$, i.e.~the set
\[
\graph(u\vert_{\Omega\setminus \Sigma} = \left\{ (x, y) \; :\; 
x \in \Omega\setminus\Sigma,\; y \in \supp(u(x)) \right\},
\]
is a smoothly embedded $m$-dimensional submanifold of $\R^{m+n}$.
\end{itemize}
\end{theorem}

\begin{remark}
We refer to \cite{DS1, Sp10} for the proofs and
more refined results in the case of of two dimensional domains.
Moreover, for some results concerning the boundary regularity we refer to \cite{Hirsch14}, and for an improved estimate of the singular set to \cite{FoMaSp}.
\end{remark}

\medskip

We close this section by some considerations on the $Q$-valued functions.
For the reasons explained in the previous section, a $Q$-valued function has to
be considered as an intrinsic map taking values in the non-smooth space of $Q$-points
$\Iq$, and cannot be reduced to a ``superposition'' of a number $Q$ of functions.
Nevertheless, in many situations it is possible to handle
$Q$-valued functions as a superposition.
For example, as shown in \cite[Proposition~0.4]{DS1} every measurable function 
$f: \R^m\to \Iqs$ can be written (not uniquely!) as 
\begin{equation}\label{e:selection}
f(x)=\sum_{i=1}^Q \a{f_i(x)}\quad\text{for $\cH^m$-a.e. }x,
\end{equation}
with $f_1,\ldots,f_Q : \R^m \to \R^n$ measurable functions.

Similarly, for weakly differentiable functions it is possible to define a notion
of pointwise approximate differential (cp.~\cite[Corollary~2,7]{DS1})
\[
Df = \sum_i\a{Df_i} \in \Iq (\R^{n\times m}),
\]
with the property that at almost every $x$ it holds $Df_i(x) = Df_j(x)$ if $f_i(x) = f_j(x)$.
Note, however, that the functions $f_i$ do not need to be weakly differentiable in \eqref{e:selection},
for the $Q$-valued function $f$ has an approximate differential.

\subsection{Graph of Lipschitz $Q$-valued functions}
There is a canonical way to give the structure of integer rectifiable
currents to the graph of a Lipschitz $Q$-valued function.

To this aim, we consider \textit{proper} $Q$-valued functions,
i.e.~measurable functions $F:M \to \Iq(\R^{m+n})$
(where $M$ is any $m$-dimensional submanifold of $\R^{m+n}$)
such that there is a measurable selection $F = \sum_i \a{F_i}$ for which
\[
\bigcup_i \overline{(F_i)^{-1}(K)}
\]
is compact for every compact $K \subset \R^{m+n}$.
It is then obvious that if there exists such a selection,
then {\em every} measurable selection shares the same property.

By a simple induction argument (cp.~\cite[Lemma~1.1]{DS2}),
there are a countable partition of $M$ in bounded measurable subsets $M_i$ ($i\in \N$)
and Lipschitz functions $f^j_i: M_i\to \R^{m+n}$ ($j\in \{1, \ldots, Q\}$) such that
\begin{itemize}
\item[(a)] $F|_{M_i} = \sum_{j=1}^Q \a{f^j_i}$ for every $i\in \N$ and $\Lip (f^j_i)\leq
\Lip (F)$ $\forall i,j$;
\item[(b)] $\forall\; i\in \N$ and $j, j' \in \{1, \ldots ,Q\}$,
either $f_i^j \equiv f_i^{j'}$ or
$f_i^j(x) \neq f_i^{j'}(x)$ $\forall x \in M_i$;
\item[(c)] $\forall \;i$ we have $DF (x) = \sum_{j=1}^Q \a{Df_i^j (x)}$ for a.e. $x\in M_i$.
\end{itemize}

We can then give the following definition.

\begin{definition}[$Q$-valued push-forward]\label{d:push_forward}
Let $M$ be an oriented submanifold of $\R^{m+n}$ of dimension $m$
and let $F:M\to\Iq(\R^{m+n})$ be a proper Lipschitz map.
Then, we define the push-forward of $M$ through $F$ as the current
\[
\mathbf{T}_F = \sum_{i,j} (f_i^j)_\sharp \a{M_i},
\]
where $M_i$ and $f_i^j$ are as above: that is,
\begin{equation}\label{e:push_forward}
\mathbf{T}_F (\omega) := \sum_{i\in\N} \sum_{j=1}^Q \int_{M_i} 
\langle \omega (f_i^j(x)), Df_i^j (x)_\sharp \vec{e}(x)\,\rangle \, d\cH^m(x)
\quad \forall\; \omega \in \sD^m(\R^{n})\, .
\end{equation}
\end{definition}

One can prove that the current in Definition~\ref{d:push_forward} does
not depend on the decomposition chosen for $M$ and $f$, and moreover
is integer rectifiable (cp.~\cite[Proposition~1.4]{DS2})

A particular class of push-forwards are given by graphs.

\begin{definition}[$Q$-graphs]\label{d:Q-graphs}
Let $f= \sum_i \a{f_i} : \R^m \to \Iqs$ be Lipschitz and define 
the map $F: M \to \Iq(\R^{m+n})$ as $F (x):= \sum_{i=1}^Q \a{(x, f_i (x))}$.
Then, $\bT_F$ is  the {\em current associated to the graph $\gr (f)$} and will be denoted by $\bG_f$. 
\end{definition}

\medskip

The main result concerning the push-forward of a $Q$-valued function
is the following (see \cite[Theorem~2.1]{DS2}).

\begin{theorem}[Boundary of the push-forward]\label{t:commute}
Let $M\subset\R^{m+n}$ be an $m$-dimensional submanifold with boundary,
$F:M \to \Iq(\R^{m+n})$ a proper Lipschitz function and $f= F\vert_{\de M}$.
Then, $\partial \mathbf{T}_F = \mathbf{T}_f$.
\end{theorem}

\medskip

Moreover, the following Taylor expansion of the mass of
a graph holds (cp.~\cite[Corollary~3.3]{DS2}).

\begin{proposition}[Expansion of $\mass (\bG_f)$] \label{c:taylor_area}
There exist dimensional constants $\bar c, C>0$ 
such that, if $\Omega\subset \R^m$ is a bounded open set
and $f:\Omega \to \Iqs$ is a Lipschitz map with $\Lip (f)\leq \bar{c}$, then
\begin{equation}\label{e:taylor_grafico}
\mass (\mathbf{G}_f) = Q |\Omega| + \frac{1}{2} \int_\Omega |Df|^2 + \int_\Omega \sum_i \bar{R}_4 (Df_i)\, ,
\end{equation}
where $\bar{R}_4\in C^1(\R^{n\times m})$
satisfies $|\bar{R}_4 (D)|= |D|^3\bar{L} (D)$ for $\bar{L}:\R^{n\times m} \to \R$
Lipschitz with $\Lip (\bar{L})\leq C$ and
$\bar L(0) = 0$.
\end{proposition}

\subsection{Approximation of area minimizing currents}\label{ss:approx}
Finally we recall some results on the approximation of area minimizing
currents.

To this aim we need to introduce more notation.
We consider cylinders in $\R^{m+n}$ of the form
$\bC_s(x) := \bar B_s(x) \times \R^n$ with $x \in \R^m$.

Since we are interested in interior regularity, we can assume for the purposes
of this section that we are always in the following setting:
for some open cylinder $\bC_{4r} (x)$ (with $r\leq 1$)
and some positive integer $Q$, the area minimizing current $T$
has compact support in $\bC_{4r}(x)$ and satisfies
\begin{equation}\label{e:(H)}
\p_\sharp T = Q\a{\bar B_{4r} (x)}\quad\mbox{and}\quad
\de T \res \bC_{4r} (x) =0,
\end{equation}
where $\p:\R^{m+n} \to \pi_0:=\R^m\times\{0\}$ is the orthogonal projection.

We introduce next the main regularity parameter for area minimizing currents, namely
the \textit{Excess}.

\begin{definition}[Excess measure]\label{d:excess}
For a current $T$ as above we define the \textit{cylindrical excess}
$\bE(T,\bC_{r} (x))$ as follows:
\begin{align*}
\bE(T,\bC_r (x))&:= \frac{\|T\| (\bC_r (x))}{\omega_m r^m} - Q\\
&= \frac{1}{2\,\omega_m r^m} \int_{\|T\| (\bC_r (x))}|\vec T-\vec \pi_0|^2\,\d\|T\|,
\end{align*}
where $\omega_m$ is the measure of the $m$-dimensional unit ball, and
$\vec \pi_0$ is the $m$-vector orienting $\pi_0$.
\end{definition}

\medskip 

The most general approximation result of area minimizing currents is the one due
to Almgren, and reproved in \cite{DS3} with more refined techniques, which asserts
that under suitable smallness condition of the excess, an area
minimizing current coincides on a big set with a graph of
a Lipschitz $Q$-valued function.

\begin{theorem}[Almgren's strong approximation]\label{t:approx}
There exist constants $C, \gamma_1,\eps_1>0$ (depending on $m,n,Q$)
with the following property. Assume that $T$ is area minimizing in
the cylinder $\bC_{4r} (x)$ and assume that
\[
E :=\bE(T,\bC_{4\,r} (x)) < \eps_1.
\]
Then, there exist a map $f: B_r (x) \to \Iqs$
and a closed set $K\subset \bar B_r (x)$ such that the following holds:
\begin{gather}\label{e:main(i)}
\Lip (f) \leq C E^{\gamma_1}, \\
\label{e:main(ii)}
\bG_f\res (K\times \R^n)=T\res (K\times\R^{n})\quad\mbox{and}\quad
|B_r (x)\setminus K| \leq
 C \, E^{1+\gamma_1} \, r^m,\\
\label{e:main(iii)}
\left| \|T\| (\bC_{r} (x)) - Q \,\omega_m\,r^m -
{\textstyle{\frac{1}{2}}} \int_{B_{r} (x)} |Df|^2\right| \leq
 C \, E^{1+\gamma_1} \, r^m.
\end{gather}
\end{theorem}

The most important improvement of the theorem above
with respect to the preexisting approximation results is the small
power $E^{\gamma_1}$ in the three estimates
\eqref{e:main(i)} - \eqref{e:main(iii)}.
Indeed, this will play a crucial role in the construction of the center manifold.
It is worthy mentioning that, when $Q=1$ and $n=1$,
this approximation theorem was first proved with different techniques
by De Giorgi in \cite{DG} (cp.~also \cite[Appendix]{DS-cm}).

As a byproduct of this approximation, we also obtain the analog of
the so called \textit{harmonic approximation}, which allows us
to compare the Lipschitz approximation above with a Dir-minimizing function.

\begin{theorem}[Harmonic approximation]\label{t:harmonic_final}
Let $\gamma_1,\eps_1$ be the constants of Theorem \ref{t:approx}.
Then, for every $\bar{\eta}>0$, there is a positive constant $\bar{\eps}_1 < \eps_1$
with the following property. 
Assume that $T$ is as in Theorem~\ref{t:approx} and
\[
E := \bE(T,\bC_{4\,r}(x)) < \bar{\eps}_1.
\]
If $f$ is the map in Theorem~\ref{t:approx},
then there exists a $\D$-minimizing function $w$
such that
\begin{multline}\label{e:harmonic_final}
r^{-2} \int_{B_r (x)} \cG (f, w)^2 + \int_{B_r(x)} \left(|Df|-|Dw|\right)^2 \\
+ \int_{B_r(x)} |D(\etaa\circ f)-
D(\etaa \circ w)|^2\leq \bar{\eta}\, E \, r^m,
\end{multline}
where $\etaa : \Iqs \to \R^n$ is the average map
\[
\etaa\left(\sum_i\a{P_i}\right) = \frac{1}{Q} \sum_i P_i.
\]
\end{theorem}

%
%
\section{Selection of contradiction's sequence}
\addtocontents{toc}{\protect\setcounter{tocdepth}{1}}

In this section we give the details of the first step (A) in \S~\ref{ss:sketch},
namely the selection of a common subsequence such that the rescaled currents converge
to a flat tangent cone and the measure of the singular set remains uniformly
bounded below away from zero.
For this purpose, we introduce the following notation.
We denote by $\B_r(x)$ the open ball of radius $r>0$ in $\R^{m+n}$
(we do not write the point $x$ if the origin)
and, for $Q\in \N$, we denote by $\rD_Q(T)$
the points of density $Q$ of the current $T$,
and set
\begin{gather*}
\Reg_Q (T) := \Reg (T)\cap \rD_Q (T) \quad\text{and}\quad
\Sing_Q (T) := \Sing (T)\cap \rD_Q (T).
\end{gather*}

The precise properties of the sequence that will be used in the blowup
argument are stated in the following proposition.
We recall that the main hypothesis at the base of the proof is the
contradiction assumption of \S~\ref{ss:sketch}, which we restate for
reader's convenience.

\medskip

\textbf{Contradiction assumption:} there exist numbers
$m\geq 2$, $n \in \N$, $\alpha>0$
and an area minimizing $m$-dimensional integer rectifiable current
$T$ in $\R^{m+n}$ such that
\[
\cH^{m-2+\alpha}(\Sing(T)) > 0.
\]

\medskip

We introduce the \textit{spherical excess} defined as follows: for a
given $m$-dimensional plane $\pi$,
\begin{gather*}
\bE(T,\B_r (x),\pi) := \frac{1}{2\,\omega_m\,r^m}\int_{\B_r (x)} |\vec T - \vec \pi|^2 \, d\|T\|,\\
\bE(T,\B_r (x)):=\min_\tau \bE (T, \B_r (x), \tau) .
\end{gather*}

\begin{proposition}[Contradiction's sequence]\label{p:seq}
Under the contradiction assumption,
there exist
\begin{enumerate}
\item constants $m,n, Q \geq 2$ natural numbers and $\alpha,\eta>0$
real numbers;

\item an $m$-dimensional
area minimizing integer rectifiable
current $T$ in $\R^{m+n}$ with $\de T = 0$;

\item a sequence $r_k\downarrow 0$
\end{enumerate}
such that $0\in \rD_Q (T)$ and the following holds:
\begin{gather}
\lim_{k\to+\infty}\bE(T_{0,r_k}, \B_{10}) = 0,\label{e:seq1}\\
\lim_{k\to+\infty} \cH^{m-2+\alpha}_\infty (\rD_Q (T_{0,r_k}) \cap \B_1) > \eta,\label{e:seq2}\\
\cH^m \big((\B_1\cap \supp (T_{0, r_k}))\setminus \rD_Q (T_{0,r_k})\big) > 0\quad \forall\; k\in\N.\label{e:seq3}
\end{gather}
\end{proposition}
Here $\cH^{m-2+\alpha}_\infty$ is the Hausdorff premeasure computed without any
restriction on the diameter of the sets in the coverings. 

By Almgren's stratification theorem and by general measure theoretic arguments,
there exist sequences satisfying either \eqref{e:seq1} or \eqref{e:seq2}.
The two subsequences might, however, be different: we show
the existence of one point and a single subsequence along which
\textit{both} conclusions hold.
The proof of the proposition is based on the following two results.

\begin{theorem}[Almgren {\cite[2.27]{Alm}}]\label{t:strat}
Let $\alpha>0$ and let $T$ be
an integer rectifiable area minimizing current in $\R^{m+n}$.
Then,
\begin{itemize}
\item[(1)] for $\cH^{m-2+\alpha}$-a.e.~point
$x \in \supp(T) \setminus \supp(\de T)$
there exists a subsequence $s_k\downarrow 0$ such that
$T_{x,s_k}$ converges to a flat cone;
\item[(2)] for $\cH^{m-3+\alpha}$-a.e.~point
$x \in \supp(T) \setminus \supp(\de T)$, it holds that $\Theta(T,x) \in \Z$.
\end{itemize}
\end{theorem}

\begin{lemma}\label{l:regular cone}
Let $S$ be an $m$-dimensional area minimizing integral current, which is a cone
in $\R^{m+n}$ with $\partial S=0$,
$Q= \Theta (S, 0) \in \mathbb N \setminus \{0\}$, and assume that
\begin{gather*}
\cH^m\big(D_Q(S)) >0 \quad\text{and}\quad
\cH^{m-1}(\Sing_Q(S))=0.
\end{gather*}
Then $S$ is an $m$-dimensional plane with multiplicity $Q$.
\end{lemma}

\begin{proof}[Proof of Proposition~\ref{p:seq}] Let $m>1$
be the smallest integer for which Theorem \ref{t:main} fails.
In view of Almgren's stratification Theorem~\ref{t:strat},
we can assume that there exist
an integer rectifiable area minimizing current $R$ of dimension $m$
and a positive integer $Q$
such that the Hausdorff dimension of $\Sing_Q (R)$ is larger than $m-2$.
We fix the smallest $Q$ for which such a current $R$ exists and note that
by Allard's regularity theorem (cp.~\cite{All}) it must be $Q>1$.

Let $\alpha>0$ be such that $\cH^{m-2+\alpha}(\Sing_Q(R))>0$, and consider a density
point $x_0$ for the measure $\cH^{m-2+\alpha}$
(without loss of generality $x_0=0$). In particular,
there exists $r_k\downarrow 0$ such that
\[
\lim_{k\to+\infty} \frac{\cH^{m-2+\alpha}_\infty \big(\Sing_Q (R)\cap \B_{r_k}\big)}{r_k^{m-2+\alpha}} > 0\, .
\]
Up to a subsequence (not relabeled) we can assume that $R_{0,r_k}\to S$, with
$S$ a tangent cone. 
If $S$ is a multiplicity $Q$ flat plane, then we set $T:=R$ and 
the proposition is proven (indeed, \eqref{e:seq3} is satisfied because $0\in\Sing (R)$ and $\|R\|\geq \cH^m \res \supp (R)$).

If $S$ is {\em not} flat, taking into account the convergence properties of area
minimizing currents \cite[Theorem~34.5]{Sim} and the upper semicontinuity of
$\cH_\infty^{m-2+\alpha}$ under the Hausdorff convergence of compact sets,
we deduce
\begin{equation}\label{e:upper_density}
\cH^{m-2+\alpha}_\infty \big(\rD_Q (S) \cap \bar{\B}_1\big) \geq
\liminf_{k\to+\infty} \cH^{m-2+\alpha}_\infty \big(\rD_Q (R_{0, r_k})\cap \bar{\B}_1\big) > 0.
\end{equation}
We claim that \eqref{e:upper_density} implies
\begin{equation}\label{e:persistence sing}
\cH^{m-2+\alpha}_\infty (\Sing_Q (S)) > 0.
\end{equation}
Indeed, if all points of $\rD_Q (S)$ are singular, then \eqref{e:persistence sing}
follows from
\eqref{e:upper_density} directly.
Otherwise, $\Reg_Q(S)$ is not empty, thus implying $\cH^m (D_Q (S) \cap \B_1)>0$:
we can then apply Lemma~\ref{l:regular cone} and infer that,
since $S$ is not regular, then
$\cH^{m-1} (\Sing_Q (S))>0$ and \eqref{e:persistence sing} holds.

We can, hence, find $x\in \Sing_Q (S)\setminus \{0\}$ and $r_k\downarrow 0$ such that
\[
\lim_{k\to+\infty} \frac{\cH^{m-2+\alpha}_\infty \big(\Sing_Q (S)\cap \B_{r_k}(x))}{r_k^{m-2+\alpha}} > 0.
\]
Up to a subsequence (not relabelled), we can assume that $S_{x, r_k}$ converges to $S_1$.
Since $S_1$ is a tangent cone to the cone $S$ at $x\neq 0$, $S_1$ splits off a line,
i.e.~$S_1 = S_2\times \a{\{ t\, e : t \in \R\} }$ for some $e \in \mathbb{S}^{m+n-1}$,
for some area minimizing cone $S_2$ in $\R^{m-1+n}$
and some $v\in \R^{m+n}$ (cp.~\cite[Lemma 35.5]{Sim}).
Since $m$ is, by assumption, the smallest integer for which Theorem \ref{t:main} fails, $\cH^{m-3+\alpha} (\Sing (S_2)) = 0$ and, hence, $\cH^{m-2+\alpha} (\Sing_Q (S_1))= 0$.
On the other hand, arguing as for \eqref{e:upper_density}, we have
\[
\cH^{m-2+\alpha}_\infty (\rD_Q (S_1)\cap \bar\B_1)\geq
\limsup_{k\to+\infty} \cH^{m-2+\alpha}_\infty (\rD_Q (S_{x, r_k})\cap \bar\B_1)  > 0.
\]
Thus $\Reg_Q (S_1)\neq \emptyset$ and, hence, $\cH^m (D_Q (S_1))> 0$.
We then can apply Lemma~\ref{l:regular cone} again and conclude that $S_1$ is an $m$-dimensional plane with multiplicity $Q$. 
Therefore, the proposition follows taking $T$ a suitable translation of $S$.
\end{proof}

\subsection*{Proof of Lemma~\ref{l:regular cone}}

We premise the following lemma.

\begin{lemma}\label{l:separate}
Let $T$ be an integer rectifiable current of dimension $m$
in $\R^{m+n}$ with locally finite mass and $U$ an open set such that
\[
\cH^{m-1}(\partial U \cap \supp (T)) = 0
\quad\text{and}\quad (\partial T) \res U = 0.
\]
Then $\partial (T\res U)=0$.
\end{lemma}

\begin{proof} Consider $V\subset \subset \R^{m+n}$.
By the slicing theory
\[
S_r:= T\res (V \cap U \cap \{{\rm dist}\, (x, \partial U) >r\})
\]
is a normal current in ${\bf N}_m (V)$ for a.e.~$r$.
Since
\[
\mass (T\res(V\cap U)-S_r)\to 0 \quad\text{as}\quad r\downarrow 0,
\]
we conclude that
$T\res (U\cap V)$
is in the $\mass$-closure of ${\bf N}_m (V)$. Thus, by \cite[4.1.17]{Fed},
$T\res U$ is a flat chain in $\R^{m+n}$ and by \cite[4.1.12]{Fed}
$\partial (T\res U)$ is a flat chain.
Since $\supp (\partial (T\res U)) \subset \partial U\cap \supp (T)$, we can
apply \cite[Theorem 4.1.20]{Fed} to conclude that $\partial (T\res U)=0$.
\end{proof}

We next prove Lemma~\ref{l:regular cone}.
For each $x\in \Reg_Q (S)$, let  $r_x$ be such that $S\res \B_{2r_x} (x) = Q \a{\Gamma}$
for some regular submanifold $\Gamma$ and set
\[
U := \bigcup_{x \in \Reg_Q(S)} \B_{r_x} (x).
\]
Obviously, $\Reg_Q (S)\subset U$; hence, by assumption, it is not empty.
Fix $x\in \supp (S)\cap \partial U$. 
Let next $(x_k)_{k \in \N}\subset \Reg_Q (S)$ be such that
\[
{\rm dist}\, (x, \B_{r_{x_k}}(x_k)) \to 0.
\]
We necessarily have that $r_{x_k}\to 0$: otherwise we would have
$x\in \B_{2r_{x_k}} (x_k)$ for some $k$, which would imply $x\in \Reg_Q (S)\subset U$, i.e. a contradiction.
Therefore, $x_k\to x$ and,
by \cite[Theorem~35.1]{Sim},
\[
Q=\limsup_{k\to+\infty}\Theta(S,x_k)  \leq \Theta (S, x) = \lim_{\lambda \downarrow 0} \Theta (S, \lambda x) \leq \Theta(S,0) = Q.
\]
This implies $x\in \rD_Q (S)$.
Since $x\in \partial U$, we must then have $x\in \Sing_Q (S)$.
Thus, we conclude that
$\cH^{m-1} (\supp (S)\cap \de U) = 0$.
It follows from Lemma~\ref{l:separate} that
$S':=S\res U$ has $0$ boundary in $\R^{m+n}$.
Moreover, since $S$ is an area minimizing cone, $S'$ is also an area-minimizing cone. 
By definition of $U$ we have $\Theta (S', x) = Q$ for $\|S'\|$-a.e.~$x$ and, by semicontinuity,
\[
Q\leq\Theta (S', 0)\leq \Theta (S, 0) =Q.
\]
We apply Allard's theorem \cite{All} and deduce that $S'$ is regular,
i.e.~$S'$ is an $m$-plane with multiplicity $Q$. Finally, from $\Theta (S', 0) = \Theta (S, 0)$, we infer
$S'=S$. \qed

%
%
\section{Center manifold's construction}
\addtocontents{toc}{\protect\setcounter{tocdepth}{2}}

In this section we describe the procedure for the construction of the center manifold.
As mentioned in the introduction, this is the most complicated part of the proof:
indeed, the construction of the center manifold comes together with a series of
other estimates which will enter significantly in the proof of
the main Theorem~\ref{t:main}.
In particular, as an outcome of the procedure we obtain
the following several things.

\begin{itemize}
\item[(1)] A decomposition of the horizontal plane $\pi_0 = \R^m \times \{0\}$
of ``Whitney's type''.

\item[(2)] A family of interpolating functions defined on the cubes of this decomposition.

\item[(3)] A normal approximation taking values in the normal bundle
of the center manifold.

\item[(4)] A set of criteria (which will in fact determine the Whitney decomposition)
which lead to what we call \textit{splitting-before-tilting} estimates.

\item[(5)] An family of intervals, called \textit{intervals of flattening}, where
the construction will be effective.

\item[(6)] A family of pairs cube--ball transforming the estimates on
the Whitney decomposition into estimates on balls (thus 
passing from the cubic lattice of the decomposition to the standard geometry of balls). 
\end{itemize}

\subsection{Notation and assumptions}
Let us recall the following notation.
Given an integer rectifiable current $T$ with compact support,
we consider the \textit{spherical}
and the \textit{cylindrical} excesses defined as follows, respectively:
for given $m$-planes $\pi, \pi'$, we set
\begin{align}
\bE(T,\B_r (x),\pi) &:= \left(2\omega_m\,r^m\right)^{-1}\int_{\B_r (x)} |\vec T - \vec \pi|^2 \, d\|T\|,\\
\qquad \bE (T, \bC_r (x, \pi), \pi') &:= \left(2\omega_m\,r^m\right)^{-1} \int_{\bC_r (x, \pi)} |\vec T - \vec \pi'|^2 \, d\|T\|\, ,
\end{align}
where $\bC_r (x, \pi) = \bar B_r(x,\pi) \times \pi^\perp$ is the cylinder over
the closed ball $\bar B_r(x,\pi)$ or radius $r$ and center $x$ in the $m$-dimensional plane $\pi$.
And we consider the {\it height function} in a set $A$ (we denote by
$\p_\pi$ the orthogonal projection on a plane $\pi$)
\[
\bh(T,A,\pi) := \sup_{x,y\,\in\,\supp(T)\,\cap\, A} |\p_{\pi^\perp}(x)-\p_{\pi^\perp}(y)|\, .
\]
We also set
\begin{equation}\label{e:optimal_pi}
\bE(T,\B_r (x)):=\min_\tau \bE (T, \B_r (x), \tau) = \bE(T,\B_r (x),\pi),
\end{equation} 
and we will use $\bE (T, \bC_r (x, \pi))$ in place of
$\bE(T, \bC_r (x, \pi), \pi)$: note that it coincides with
the cylindrical excess as defined in \S~\ref{ss:approx}
when
\[
(\p_\pi)_\sharp T \res \bC_r (x, \pi)=
Q \a{\bar B_r (\p_\pi (x), \pi)}.  
\]

\medskip

In this section we will work with an area minimizing integer rectifiable current
$T^0$ with compact support 
which satisfies the following assumptions: for some constant $\eps_2\in (0,1)$,
which we always suppose to be small enough,
\begin{gather}
\Theta (0, T^0) = Q\quad \mbox{and}\quad \partial T^0 \res \B_{6\sqrt{m}} = 0,\label{e:basic}\\
\quad \|T^0\| (\B_{6\sqrt{m} \rho}) \leq \big(\omega_m Q (6\sqrt{m})^m + \eps_2^2\big)\,\rho^m
\quad \forall \rho\leq 1,\label{e:basic2}\\
E:=\bE\left(T^0,\B_{6\sqrt{m}}\right)=\bE\left(T^0,\B_{6\sqrt{m}},\pi_0\right)
\leq \eps_2^2,\label{e:pi0_ottimale}
\end{gather}

It follows from standard considerations in geometric measure theory that
there are positive constants $C_0 (m,n,Q)$ and $c_0 (m,n, Q)$ with the following property.
If $T^0$ is as in \eqref{e:basic} - \eqref{e:pi0_ottimale},
$\eps_2 < c_0$ and $T:= T^0 \res \B_{23\sqrt{m}/4}$, then:
\begin{gather}
\partial T \res \bC_{11\sqrt{m}/2} (0, \pi_0)= 0,\label{e:semplice 0}\\
(\p_{\pi_0})_\sharp T\res \bC_{11 \sqrt{m}/2} (0, \pi_0) = Q \a{B_{11\sqrt{m}/2} (0, \pi_0)},\label{e:geo semplice 1}\\
\bh (T, \bC_{5\sqrt{m}} (0, \pi_0)) \leq C_0 \eps_2^{\frac{1}{m}}\, .\label{e:pre_height}
\end{gather}
In particular for each $x\in B_{11\sqrt{m}/2} (0, \pi_0)$ there is  a point
$p\in \supp (T)$ with $\p_{\pi_0} (p)=x$.

\subsection{Whitney decomposition and interpolating functions}
The construction of the center manifold is done
by following a suitable decomposition of the horizontal plane $\pi_0$ into cubes.
We denote by $\sC^j$, $j\in \N$, the family of dyadic
closed cubes $L$ of $\pi_0$ with side-length $2^{1-j}=:2\,\ell (L)$.
Next we set $\sC := \bigcup_{j\in \N} \sC^j$. 
If $H$ and $L$ are two cubes in $\sC$ with $H\subset L$, then we call $L$ an {\em ancestor} of $H$ and $H$ a {\em descendant} of $L$. When in addition $\ell (L) = 2\ell (H)$, $H$ is {\em a son} of $L$ and $L$ {\em the father} of $H$.

\begin{definition}\label{e:whitney} A Whitney decomposition of $[-4,4]^m\subset \pi_0$ consists of a closed set $\bGam\subset [-4,4]^m$ and a family $\mathscr{W}\subset \sC$ satisfying the following properties:
\begin{itemize}
\item[(w1)] $\bGam \cup \bigcup_{L\in \mathscr{W}} L = [-4,4]^m$ and $\bGam$ does not intersect any element of $\mathscr{W}$;
\item[(w2)] the interiors of any pair of distinct cubes $L_1, L_2\in \mathscr{W}$ are disjoint;
\item[(w3)] if $L_1, L_2\in \mathscr{W}$ have nonempty intersection, then
\[
\frac{1}{2}\ell (L_1) \leq \ell (L_2) \leq 2\, \ell (L_1).
\]
\end{itemize}
\end{definition}

Observe that (w1) - (w3) imply
\[
{\rm dist}\, (\bGam, L) := \inf \big\{ |x-y|: x\in L, y\in \bGam \big\} \geq 2\ell (L)
\quad \text{for every $L\in \mathscr{W}$.}
\]
However, we do {\em not} require any inequality of the form 
${\rm dist}\, (\bGam, L) \leq C \ell (L)$, although this would be customary for what is commonly 
called Whitney decomposition in the literature.

\medskip

We denote by $\sS^j$ all the dyadic cubes with side-length $2^{1-j}$
which are not contained in $\sW$ and set $\sS:= \cup_{j\geq N_0} \sS^j$
for some big natural number $N_0$.
For each cube $L \in \sW \cup \sS$, we set $r_L= M_0 \sqrt{m}\ell(L)$,
with $M_0 \in \N$ a dimensional constant to be fixed later, and we call its
center $x_L$.
We can then find points $p_L\in \supp(T)$,
with coordinates $p_L=(x_L, y_L) \in \pi_0 \times \pi_0^\perp$,
and \textit{interpolating functions}
\[
g_L : B_{4r_L}(p_L, \pi_0) \to \pi_0^\perp,
\]
such that the following holds: for every $H, L \in \sW\cup \sS$,
\begin{gather}
\|g_H\|_{C^0}\leq C\, E^{\frac{1}{2m}} \quad\text{and}\quad
\|Dg_H\|_{C^{2, \kappa}} \leq C E^{\frac{1}{2}};\label{e:g1}\\
\|g_H-g_L\|_{C^i (B_{r_L} (p_L,\pi_0))} \leq C E^{\frac{1}{2}} \ell (H)^{3+\kappa-i}\label{e:g2}\\
\text{ $\forall\;i\in \{0, \ldots, 3\}$ if $H\cap L\neq \emptyset$};\notag\\
|D^3 g_H (x_H) - D^3 g_L (x_L)| \leq C E^{\frac{1}{2}} |x_H-x_L|^\kappa;\label{e:g3}\\
\sup_{(x, y) \in\supp(T),\;x \in H}\|g_H-y\|_{C^0} \leq C E^{\frac{1}{2m}} \ell (H),\label{e:g4}
\end{gather}
for some $\kappa>0$, and where we used the notation
\[
B_r(p_L, \pi_0) := \B_r(p_L) \cap (p_L + \pi_0).
\]

\medskip

It is now very simple to show how to patch all the interpolating functions
$g_L$ in order to construct a center manifold.
To this aim, we set
\[
\sP^j := \sS^j \cup \left\{L \in \sW\;:\;\,\ell(L)\geq 2^{-j}\right\}.
\]
For every $L\in \sP^j$ we define 
\[
\vartheta_L (y):= \vartheta \left(\frac{y-x_L}{\ell (L)}\right),
\]
for some fixed $\vartheta\in C^\infty_c \big([-\frac{17}{16}, \frac{17}{16}]^m, [0,1]\big)$ that is identically $1$ on $[-1,1]^m$.
We can then patch all the interpolating functions using the partition
of the unit induced by the $\vartheta_L$, i.e.
\begin{equation}\label{e:ph_j}
\varphi_j := \frac{\sum_{L\in \sP^j} \vartheta_L g_L}{\sum_{L\in \sP^j} \vartheta_L}.
\end{equation}

The following theorem is now a very easy consequence of the estimates
on the interpolating functions.

\begin{theorem}[Existence of the center manifold]\label{t:cm}
Assume to be given a Whitney decomposition $(\Gamma, \sW)$
and interpolating functions $g_H$ as above.
If $\eps_2$ is sufficiently small, then
\begin{itemize}
\item[(i)] the functions $\varphi_j$ defined in \eqref{e:ph_j} satisfy
\[
\|D\varphi_j\|_{C^{2, \kappa}} \leq C E^{\frac{1}{2}}
\quad\text{and}\quad\|\varphi_j\|_{C^0}
\leq C E^{\frac{1}{2m}},
\]

\item[(ii)] $\varphi_j$ converges to a map $\phii$ such that
$\cM:= \gr (\phii|_{]-4,4[^m})$ is a $C^{3,\kappa}$ submanifold of $\Sigma$,
called in the sequel \textup{center manifold},

\item[(iii)] for all $x \in \bGam$, the point $(x, \phii(x)) \in \supp(T)$
and is a multiplicity $Q$ point.
Setting $\Phii(y) := (y,\phii(y))$, we call $\Phii (\bGam)$ the {\em contact set}.
\end{itemize}
\end{theorem}

\begin{proof}
Define $\chi_H := \vartheta_H/ (\sum_{L\in \sP^j} \vartheta_L)$ 
and observe that 
\begin{align}
\sum \chi_H = 1 \quad \text{and}\quad
\|\chi_H\|_{C^i} &\leq C_0 (i, m, n) \,\ell (H)^{-i} \qquad \forall i\in \N\label{e:p_unita}\, .
\end{align}
Set $\sP^j (H):=\{L\in \sP^j : L\cap H\neq\emptyset\}\setminus \{H\}$. 
By construction
\[
\frac{1}{2} \ell (L) \leq \ell (H) \leq 2\, \ell (L)
\quad\text{for every $L\in \sP^j (H)$,}
\]
and the cardinality of $\sP^j(H)$ is bounded by a geometric constant $C_0$. 
The estimate $|\varphi_j| \leq C E^{\frac{1}{2m}}$ follows then easily 
from \eqref{e:g1}.
 
For $x\in H$ we write 
\begin{align}
\varphi_j (x) &=\Big(g_H \chi_H  + \sum_{L\in \sP^j (H)} 
g_L \chi_L\Big) (x) = g_H (x) + \sum_{L\in \sP^j (H)} (g_L - g_H) \chi_L\,  (x)\, .
\end{align}
Using the Leibniz rule, \eqref{e:p_unita}, \eqref{e:g1} and \eqref{e:g2},
for $i\in \{1,2,3\}$ we get
\begin{align*}
\|D^i \varphi_j\|_{C^0 (H)} &\leq \|g_H\|_{C^i} + \sum_{0\leq l \leq i} 
\sum_{L\in \sP^j (H)} \|g_L-g_H\|_{C^l (H)} \ell (L)^{l-i}\\
& \leq C E^{\frac{1}{2}} \big(1+\ell (H)^{3+\kappa-i}\big).
\end{align*}
Next, using also $[D^3 g_H - D^3 g_L]_\kappa \leq C E^{\frac{1}{2}}$, we obtain 
\begin{multline*}
[D^3  \varphi_j]_{\kappa, H} \leq \sum_{0\leq l \leq 3} \sum_{L\in \sP^j (H)} \ell (H)^{l-3} \big(
\ell (H)^{-\kappa} \|D^l (g_L-g_H)\|_{C^0 (H)} \\
\qquad+ [D^l (g_L - g_H)]_{\kappa, H}\big) + [D^3 g_H]_{\kappa, H}  \leq C E^{\frac{1}{2}}.
\end{multline*}
Fix now $x, y\in [-4,4]^m$, let $H, L\in \sP^j$ be such that $x\in H$ and $y\in L$.
If $H\cap L\neq \emptyset$, then
\begin{equation}\label{e:primo caso}
|D^3 \varphi_j (x) - D^3 \varphi_j (y)| \leq C \big([D^3 \varphi_j]_{\kappa, H}
+ [D^3 \varphi_j]_{\kappa, L}\big) |x-y|^{\kappa}.
\end{equation}
If $H\cap L= \emptyset$, we assume without loss of generality
 $\ell (H) \leq \ell (L)$ and observe that 
\[
\max \big\{|x-x_H|, |y-x_L|\big\} \leq \ell(L) \leq |x-y|\, .
\] 
Moreover, by construction ${\varphi}_j$ is identically equal to $g_H$ in a neighborhood of its center $x_H$. Thus, we can estimate 
\begin{align*}
|D^3 \varphi_j (x) - D^3 \varphi_j (y)| & \leq 
|D^3 \varphi_j (x) - D^3 \varphi_j (x_H)|  
+ |D^3 g_H (x_H) - D^3 g_L (x_L)|\\
& + |D^3 \varphi_j (x_L) - D^3 \varphi_j (y)|\nonumber\\
\leq \;&C E^{\frac{1}{2}} \left(|x-x_H|^\kappa + |x_H-x_L|^\kappa+ |y-x_L|^\kappa \right)\\
& \leq C E^{\frac{1}{2}} |x-y|^\kappa\, ,
\end{align*}
where we used \eqref{e:primo caso} and \eqref{e:g3}.
The convergence of the sequence $\varphi_j$ (up to subsequences) and (iii)
are now simple consequences of \eqref{e:g4} (details are left to the reader).
\end{proof}

\subsection{Normal approximation}
The main feature of the center manifold $\cM$ lies actually
in the fact that it allows to make a good approximation
of the current which turns out to be almost centered by $\cM$.

We introduce the following definition.

\begin{definition}[$\cM$-normal approximation]\label{d:app}
An {\em $\cM$-normal approximation} of $T$ is given by a pair $(\cK, F)$ such that
\begin{itemize}
\item[(A1)] $F: \cM\to \Iq (\bU)$ is Lipschitz and takes the special form 
\[
F (x) = \sum_i \a{x+N_i (x)},
\]
with $N_i (x)\perp T_x \cM$ for every $x\in \cM$ and $i=1,\ldots, Q$.
\item[(A2)] $\cK\subset \cM$ is closed, contains $\Phii \big(\bGam\cap [-\frac{7}{2}, \frac{7}{2}]^m\big)$ and
\[
\bT_F \res \p^{-1} (\cK) = T \res \p^{-1} (\cK).
\]
\end{itemize}
The map $N = \sum_i \a{N_i}:\cM \to \Iq (\bU)$ is
called {\em the normal part} of $F$.
\end{definition}

\medskip

As proven in \cite[Theorem~2.4]{DS4}, the center manifold $\cM$ of the previous
section allows to construct an $\cM$-normal approximation which does
approximate the area minimizing current $T$.
In order to state the result, 
to each $L\in \sW$ we associate a {\em Whitney region} $\cL$ on $\cM$ as follows:
\[
\cL := \Phii \left(H\cap [-\frac{7}{2},\frac{7}{2}]^m\right),
\]
where $H$ is the cube concentric to $L$ with $\ell (H) = \frac{17}{16} \ell (L)$.
We will use $\|N|_{\cL}\|_0$ to denote the quantity
$\sup_{x\in \cL} \cG (N (x), Q\a{0})$.

\begin{theorem}\label{t:approxN}
Let $\gamma_2 := \frac{\gamma_1}{4}$, with $\gamma_1$ the constant of Theorem~\ref{t:approx}.
Under the hypotheses of Theorem~\ref{t:cm},
if $\eps_2$ is sufficiently small, then
there exist constants $\beta_2,\delta_2>0$
and an $\cM$-normal approximation $(\cK, F)$ such that
the following estimates hold on every Whitney region $\cL$:
\begin{gather}
\Lip (N|
_\cL) \leq C E^{\gamma_2} \ell (L)^{\gamma_2} \quad\mbox{and}\quad  \|N|
_\cL\|_{C^0}\leq C E^{\frac{1}{2m}} \ell (L)^{1+\beta_2},\label{e:Lip_regional}\\
\int_{\cL} |DN|^2 \leq C E \,\ell (L)^{m+2-2\delta_2},\label{e:Dir_regional}
\\
|\cL\setminus \cK| + \|\bT_F - T\| (\p^{-1} (\cL)) \leq C E^{1+\gamma_2} \ell (L)^{m+2+\gamma_2}.\label{e:err_regional}
\end{gather}
Moreover, for any $a>0$ and any Borel set $\cV\subset \cL$, we have
\begin{multline}\label{e:av_region}
\int_\cV |\etaa\circ N| \leq 
C E \left(\ell (L)^{3+\frac{\beta_2}{3}}+ a\,\ell (L)^{2+\frac{\gamma_2}{2}}\right) |\cV|\\
 + \frac{C}{a} 
\int_\cV \cG \big(N, Q \a{\etaa\circ N}\big)^{2+\gamma_2}.
\end{multline} 
\end{theorem}

Let us briefly explain the conclusions of the theorem.
The estimates in \eqref{e:Lip_regional} and \eqref{e:Dir_regional}
concern the regularity
properties of the normal approximation $N$, and will play an
important role in many of the subsequent arguments.
However, the key properties of $N$ are \eqref{e:err_regional}
and \eqref{e:av_region}: the former
estimates the error done in the approximation on every Whitney region; while the
latter estimates
the $L^1$ norm of the average of $N$, which is a measure of
the centering of the center manifold.
Note that both estimates are in some sense ``superlinear'' with respect to the
relevant parameters: indeed, as it will be better understood later on,
they involve either a superlinear power of the excess
$E^{1+\gamma_2}$ or the $L^{2+\gamma_2}$ norm of $N$ (which is of higher order
with respect to the ``natural'' $L^2$ norm).

\subsection{Construction criteria}
The estimates and the results of the previous two subsections
depend very much on the way the Whitney decomposition, the interpolating
functions and the normal approximation are constructed.

\medskip

We start recalling the notation $p_L = (x_L, y_L)$ where $L$ is a dyadic cube,
$x_L$ its center and $y_L \in \pi_0^\perp$ is chosen in such a way
that $p_L \in \supp(T)$.
Moreover, we set
\[
\B_L := \B_{64 r_L} (p_L),
\]
where we recall that $r_L:= M_0 \sqrt{m} \,\ell (L)$ for some large constant
$M_0 \in \N$.

We define the families of cubes of the Whitney decomposition
\[
\sW = \sW_e \cup \sW_h \cup \sW_n
\quad\text{and}\quad
\sS \subset \sC.
\]
We use the notation
$\sS^j = \sS\cap \sC^j, \sW^j = \sW\cap \sC^j$ and so on.

We recall the notation for the excess,
\begin{equation*}
\bE(T,\B_r (x)):=\min_\tau \bE (T, \B_r (x), \tau) = \bE(T,\B_r (x),\pi).
\end{equation*} 
The $m$-dimensional planes $\pi$ realizing the minimum above are
called {\em optimal planes} of $T$ in a ball $\B_r (x)$ if, in addition,
$\pi$ optimizes the height among all planes that optimize the excess:
\begin{align}\label{e:optimal_pi_2}
\bh(T,\B_r(x),\pi) & = \min \big\{\bh(T,\B_r (x),\tau): \tau \mbox{ satisfies \eqref{e:optimal_pi}}\big\}\notag\\
& = : \bh(T,\B_r(x)).
\end{align}
An optimal plane in the ball $\B_L$ is denoted by $\pi_L$.

We fix a big natural number $N_0$, and constants $C_e, C_h>0$,
and we define $\sW^i = \sS^i = \emptyset $ for $i < N_0$. We proceed with $j\geq N_0$ inductively: if the father of $L\in \sC^j$ is {\em not} in $\sW^{j-1}$, then 
\begin{itemize}
\item[(EX)] $L\in \sW^j_e$ if $\bE (T, \B_L) > C_e\, E\, \ell (L)^{2-2\delta_2}$;
\item[(HT)] $L\in \sW_h^j$ if $L\not \in \mathscr{W}_e^j$ and $\bh (T, \B_L) > C_h E^{\frac{1}{2m}} \ell (L)^{1+\beta_2}$;
\item[(NN)] $L\in \sW_n^j$ if $L\not\in \sW_e^j\cup \sW_h^j$ but it intersects an element of $\sW^{j-1}$;
\end{itemize}
if none of the above occurs, then $L\in \sS^j$.

We finally set
\begin{equation}\label{e:bGamma}
\bGam:= [-4,4]^m \setminus \bigcup_{L\in \sW} L = \bigcap_{j\geq N_0} \bigcup_{L\in \sS^j} L.
\end{equation}

Observe that, if $j>N_0$ and $L\in \sS^j\cup \sW^j$, then necessarily its father belongs to $\sS^{j-1}$.

\medskip

For what concerns the interpolating functions $g_L$, they are
obtained as the result of the following procedure.

\begin{itemize}
\item[(1)] Let $L\in \sS\cup \sW$ and $\pi_L$ be an optimal plane.
Then, $T\res\bC_{32 r_L} (p_L, \pi_L)$ fulfills 
the assumptions of the approximation Theorem~\ref{t:approx}
in the cylinder $\bC_{32 r_L} (p_L, \pi_L)$, and we can then construct
a Lipschitz approximation
\[
f_L: B_{8r_L} (p_L, \pi_L)  \to \Iq (\pi_L^\perp).
\]

\item[(2)] We let 
$h_L:B_{7r_L} (p_L, \pi_L) \to \pi_L^\perp$ be a regularization of the average given by
\[
h_L:= (\etaa\circ f_L) * \varrho_{\ell (L)},
\]
where $\varrho\in C^\infty_c (B_1)$ is radial, $\int \varrho =1$ and $\int |x|^2 \varrho (x)\, dx = 0$.

\item[(3)] Finally, we find a smooth map $g_L: B_{4r_L} (p_L, \pi_0)\to \pi_0^\perp$ such that 
\[
\bG_{g_L} = \bG_{h_L}\res \bC_{4r_L} (p_L, \pi_0),
\]
where we recall that $\bG_u$ denotes the current induced by the graph
of a function $u$.
\end{itemize}

The fact that the above procedure can be applied follows from the choice of the
stopping criteria for the construction of the Whitney decomposition.
We refer to \cite{DS4} for a detailed proof. Here we only stress the fact that
this construction depends strongly on the choice of the constants involved:
in particular,
$C_e,C_h,\beta_2,\delta_2, M_0$ are positive real numbers and $N_0$ a natural number satisfying in particular
\begin{equation}\label{e:delta+beta}
\beta_2 = 4\,\delta_2 = \min \left\{\frac{1}{2m}, \frac{\gamma_1}{100}\right\},
\end{equation}
where $\gamma_1$ is the constant of Theorem~\ref{t:approx}, and
\begin{equation}
M_0 \geq C_0 (m,n,\bar{n},Q) \geq 4\,  \quad
\mbox{and}\quad \sqrt{m} M_0 2^{7-N_0} \leq 1\, . \label{e:N0}
\end{equation}
Note that $\beta_2$ and $\delta_2$ are fixed, while the other parameters are not fixed but are subject to further restrictions in the various statements, respecting a
very precise ``hierarchy'' (cp.~\cite[Assumption 1.9]{DS4}). 

\medskip

Finally, we add also a few words concerning the construction
of the normal approximation $N$.
In every Whitney region $\cL$ the map $N$ is a suitable extension of the reparametrization of the Lipschitz approximation $f_L$.
Then the estimates \eqref{e:Lip_regional}, \eqref{e:Dir_regional} and
\eqref{e:err_regional} follow easily from Theorem~\ref{t:approx}.
The most intricate proof is the one of \eqref{e:av_region} for which the choice of
the regularization $h_L$ deeply plays a role.
The main idea is that, on the optimal plane $\pi_L$, the average of the sheets of the
minimizing current is almost the graph of a harmonic function.
Therefore, a good way to regularize it (which actually would keep it unchanged
if it were exactly harmonic) is to convolve with a radial symmetric mollifier.
This procedure, which we stress is not the only possible one, will indeed preserve the main properties of the average.

\subsection{Splitting before tilting}
The above criteria are not just important for the construction purposes,
but also lead to a couple of important estimates which
will be referred to as \textit{splitting-before-tilting} estimates.
Indeed, it is not a case that the powers of the side-length in the (EX) and (HT)
criteria look like the powers in the familiar decay of the excess
and in the height bound.
In fact it turns out that, following the arguments for the height bound and for the
decay of the excess, one can infer two further consequences of the
Whitney decomposition's criteria.

\subsubsection{(HT)-cubes}
If a dyadic cube $L$ has been selected by the Whitney decomposition procedure
for the height criterion, then the $\cM$-normal approximation above the corresponding
Whitney region needs to have a large pointwise separation (see \eqref{e:s3} below).

\begin{proposition}[(HT)-estimate]\label{p:separ}
If $\eps_2$ is sufficiently small, then the following conclusions hold for every $L\in \sW_h$:
\begin{gather}
\Theta (T, p) \leq Q - \frac{1}{2}\quad\forall\; p\in \B_{16 r_L} (p_L),\label{e:s1}\\
L\cap H= \emptyset \quad \forall\;H\in \sW_n \; \text{ with }
\ell (H) \leq \frac{1}{2} \ell (L);\label{e:s2}\\
\cG \big(N (x), Q \a{\etaa \circ N (x)}\big) \geq \frac{1}{4} C_h E^{\frac{1}{2m}}
\ell (L)^{1+\beta_2}\quad \forall\;x\in \cL.\label{e:s3}
\end{gather}
\end{proposition}

A simple corollary of the previous proposition is the following.

\begin{corollary}\label{c:domains}
Given any $H\in \sW_n$ there is a chain $L =L_0, L_1, \ldots, L_j = H$ such that:
\begin{itemize}
\item[(a)] $L_0\in \sW_e$ and $L_i\in \sW_n$ for all $i=1, \ldots, j$; 
\item[(b)] $L_i\cap L_{i-1}\neq\emptyset$ and $\ell (L_i) = \frac{1}{2} \ell (L_{i-1})$ for all $i=1,\ldots, j$.
\end{itemize}
In particular,  $H\subset B_{3\sqrt{m}\ell (L)} (x_L, \pi_0)$.
\end{corollary}

We use this last corollary to partition $\sW_n$.

\begin{definition}[Domains of influence]\label{d:domains}
We first fix an ordering of the cubes in $\sW_e$ as $\{J_i\}_{i\in \mathbb N}$ so that their side-length decreases. Then $H\in \sW_n$
belongs to $\sW_n (J_0)$ if there is a chain as in Corollary \ref{c:domains} with $L_0 = J_0$.
Inductively, $\sW_n (J_r)$ is the set of cubes $H\in \sW_n \setminus \cup_{i<r} \sW_n (J_i)$ for which there is
a chain as in Corollary \ref{c:domains} with $L_0 = J_r$.
\end{definition}

\subsubsection{(Ex)-cubes}
Similarly, if a cube of the Whitney decomposition is selected by the (EX)
condition, i.e.~the excess does not decay at some given scale, then
a certain amount of separation between the sheets of the current must also
in this case occur.

\begin{proposition}[(EX)-estimate]\label{p:splitting}
If $L\in \sW_e$ and $\Omega = \Phii (B_{\ell (L)/4} (q, \pi_0))$ for some point $q\in \pi_0$ with $\dist (L, q) \leq 4\sqrt{m} \,\ell (L)$, then
\begin{align}
&C_e E \ell(L)^{m+2-2\delta_2} \leq \ell (L)^m \bE (T, \B_L) \leq C \int_\Omega |DN|^2\, ,\label{e:split_1}\\
&\int_{\cL} |DN|^2 \leq C \ell (L)^m \bE (T, \B_L) \leq C \ell (L)^{-2} \int_\Omega |N|^2\, . \label{e:split_2}
\end{align}
\end{proposition}

\medskip

Both propositions above are a typical \textit{splitting-before-tilting} phenomenon in
this sense:
the key assumption is that the excess has decayed up to a given scale (i.e.~no ``tilting'' occurs), while the conclusion is that 
a certain amount of separation between the sheets of the current (``splitting'') holds.
We borrowed this terminology from the paper by T.~Rivi{\`e}re \cite{Ri04},
where a similar phenomenon (but not completely the same) was proved
for semi-calibrated two dimensional currents as a consequence of a lower
epi-perimetric inequality.

\subsection{Intervals of flattening}
Here we define the last feature of the construction of the center manifold,
namely the so called interval of flattening.
A center manifold constitutes a good approximation of the average
of the sheets of a current as soon as the errors in Theorem~\ref{t:approxN}
are small compared to the distance from the origin.
In this case, we are forced to interrupt our blowup analysis and to
start a new center manifold. This procedure is explained in details
in the following paragraph.

\subsubsection{Defining procedure}\label{ss:flattening}
We fix the constant $c_s := \frac{1}{64\sqrt{m}}$ and notice that
$2^{-N_0} < c_s$. We set
\begin{equation}\label{e:def_R}
\mathcal{R}:=\big\{r\in ]0,1]: \bE (T, \B_{6\sqrt{m} r}) \leq \eps_3^2\big\},
\end{equation}
where $\eps_3>0$ is a suitably chosen constant, always assumed to be
smaller than $\eps_2$.
Observe that, if $(s_k)\subset \mathcal{R}$
and $s_k\uparrow s$, then $s\in \mathcal{R}$.
We cover $\cR$ with a collection $\mathcal{F}=\{I_j\}_j$ of intervals
$I_j = ]s_j, t_j]$ defined as follows: we start with 
\[
t_0:= \max \{t: t\in \mathcal{R}\}.
\]
Next assume, by induction, to have defined
\[
t_0 > s_0\geq t_1 > s_1 \geq \ldots > s_{j-1}\geq t_j,
\]
and consider the following objects:
\begin{itemize}
\item[-] $T_j := ((\iota_{0,t_{j}})_\sharp T)\res \B_{6\sqrt{m}}$, and
assume (without loss of generality, up to a rotation)
that $\bE (T_j, \B_{6\sqrt{m}}, \pi_0) = \bE(T_j,\B_{6\sqrt{m}})$;
\item[-] let $\cM_j$ the corresponding center manifold for $T_j$, given as
the graph of a map $\phii_j: \pi_0 \supset [-4,4]^m \to \pi_0^\perp$,
(for later purposes we set $\Phii_j (x) := (x, \phii_j (x))$).
\end{itemize}
Then, one of the following possibilities occurs:
\begin{itemize}
\item[(Stop)] either there is $r \in ]0, 3]$ and a cube $L$ of the
Whitney decomposition $\sW^{(j)}$ of $[-4,4]^m \subset \pi_0$ (applied to $T_j$)
such that
\begin{equation}\label{e:st}
\ell(L)\geq c_s\,r\, \qquad \mbox{and} \qquad L\cap \bar B_r (0, \pi_0)\neq \emptyset; 
\end{equation}

\item[(Go)] or there exists no radius as in (Stop). 
\end{itemize}
It is possible to show that
when (Stop) occurs for some $r$, such $r$ is smaller than $2^{-5}$.
This justifies the following:
\begin{itemize}
\item[(1)]  in case (Go) holds, we set $s_j:=0$, i.e.~$I_{j} := ]0, t_{j}]$, and end
the procedure;

\item[(2)]  in case (Stop) holds we let $s_j:= \bar r\,t_j$, where $\bar r$ is
the maximum radius satisfying (Stop). We choose then $t_{j+1}$ as the largest element
in $\mathcal{R}\cap ]0, s_j]$ and proceed iteratively.
\end{itemize}

\medskip

The following are easy consequences of the definition: for all $r\in]\frac{s_j}{t_j},3[$, it holds
\begin{gather}
\bE (T_j, \B_r )\leq C \eps_3^2 \, r^{2-2\delta_2},\\
\sup \{ \textup{dist} (x,\cM_j): x\in \supp(T_j) \cap \p^{-1}_j(\cB_r(p_j))\} \leq C\, (E^j)^\frac{1}{2m} r^{1+\beta_2},
\end{gather}
where $E^j := \bE(T_j, \B_{6\sqrt{m}})$ and $\p_j$ denotes the nearest point
projection on $\cM_j$ defined on a neighborhood of the center manifold
(for the proof we refer to \cite{DS5}).

\subsection{Families of subregions}\label{s:regions}
Let $\cM$ be a center manifold and $\Phii:\pi_0 \to \R^{m+n}$
the paremetrizing map. Set $q:= \Phii (0)$ and denote by $B$
the projection of the geodesic ball $\p_{\pi_0} (\cB_r (q))$, for some
$r\in (0,4)$.
Since $\|\phii\|_{C^{3,\kappa}}\leq C \eps_2^{1/m}$ in Theorem~\ref{t:cm},
it is simple to show that $B$ is a $C^2$ convex set and that the maximal curvature
of $\partial B$ is everywhere smaller than $\frac{2}{r}$. Thus,
for every $z\in \de B$ there is a ball $B_{r/2} (y)\subset B$ whose closure touches
$\partial B$ at $z$.

\medskip

In this section we show how one can partition the cubes of the
Whitney decomposition which intersect $B$ into disjoint families
which are labeled by pairs $(L,B(L))$ cube--ball enjoying different properties.

\begin{proposition}\label{p:covering}
There exists a set $\mathscr{Z}$ of pairs $(L,B(L))$ with this properties:
\begin{enumerate}
\item[(i)] if $(L, B(L))\in \mathscr{Z}$, then $L\in \mathscr{W}_e \cup \mathscr{W}_h$,
the radius of $B(L)$ is $\frac{\ell (L)}{4}$, $B(L)\subset B$
and ${\rm dist}\, (B(L), \partial B)\geq \frac{\ell (L)}{4}$;

\item[(ii)] if the pairs $(L, B(L)), (L', B(L'))\in \mathscr{Z}$ are distinct, then
$L$ and $L'$ are distinct and $B (L)\cap B (L') =\emptyset$;

\item[(iii)] the cubes $\mathscr{W}$ which intersect $B$
are partitioned into disjoint families $\mathscr{W} (L)$ labeled by
$(L,B(L))\in \mathscr{Z}$ such that,
if $H\in \mathscr{W} (L)$, then $H\subset B_{30 \sqrt{m} \ell (L)} (x_L)$. 
\end{enumerate}
\end{proposition}

In this way, every cube of the Whitney decomposition intersecting $B$
can be uniquely associated to a ball $B(L) \subset B$ for some
$L \in \sW_e \cap \sW_h$. This
will allow to transfer the estimates form the cubes of the Whitney decomposition
to the ball $B$.

\subsubsection{Proof of Proposition \ref{p:covering}}
We start defining appropriate families of cubes and balls.

\begin{definition}[Family of cubes]\label{d:cubes}
We first define a family $\mathcal{T}$ of cubes in the Whitney decomposition 
$\mathscr{W}$ as follows:
\begin{itemize}
\item[(i)] $\mathcal{T}$ includes all $L\in \mathscr{W}_h \cup \mathscr{W}_e$ which intersect $B$;
\item[(ii)] if $L'\in \mathscr{W}_n$ intersects $B$ and
belongs to the domain of influence
$\mathscr{W}_n (L)$ of the cube $L\in \mathscr{W}_e$ as in Definition~\ref{d:domains},
then $L\in \mathcal{T}$.
\end{itemize}
\end{definition}

It is easy to see that, if $r$ belongs to an interval of flattening, then
for every $L\in \mathcal{T}$ it holds that
$\ell (L)\leq 3 c_s r \leq r$ and ${\rm dist} (L, B) \leq 3\sqrt{m} \,\ell (L)$.
Therefore, we can also define the following associated balls.

\begin{definition}\label{d:balls}
For every $L \in \mathcal{T}$, let $x_L$ be the center of $L$ and:
\begin{itemize}
\item[(a)] if $x_L \in \overline{B}$, we then set $s(L):=\ell(L)$ and
$B^L := B_{s(L)} (x_L, \pi)$;
\item[(b)] otherwise we consider the ball $B_{r(L)} (x_L, \pi)\subset \pi$ such that its
closure
touches $\overline{B}$ at exactly one point $p(L)$, we set $s(L):= r(L) + \ell (L)$
and define $B^L:= B_{s(L)} (x_L, \pi)$.
\end{itemize}
\end{definition}
We proceed to select a countable family
$\mathscr{T}$ of pairwise disjoint balls $\{B^L\}$. We
let $S:= \sup_{L\in \mathcal{T}} s(L)$ and start selecting
a maximal subcollection $\mathscr{T}_1$ 
of pairwise disjoint balls with radii larger than $S/2$.
Clearly, $\mathscr{T}_1$ is finite. In general, at the stage $k$,
we select a maximal subcollection $\mathscr{T}_k$ of pairwise disjoint balls
which do not intersect any of the previously selected balls in
$\mathscr{T}_1 \cup \ldots \cup \mathscr{T}_{k-1}$ and which
have radii $r\in ]2^{-k} S, 2^{1-k}S]$.
Finally, we set $\mathscr{T} := \bigcup_k \mathscr{T}_k$. 

\begin{definition}[Family of pairs cube-balls $(L, B(L))\in \mathscr{Z}$]\label{d:coppie}
Recalling the convexity properties of $B$ and $\ell (L)\leq r$,
it easy to see that there exist balls
$B_{\ell (L)/4} (q_L, \pi) \subset B^L\cap B$ which lie at distance
at least $\ell (L)/4$ from $\partial B$. We denote by $B(L)$ one of such balls and 
by $\mathscr{Z}$ the collection of pairs $(L, B(L))$ with $B^L\in \mathscr{T}$.
\end{definition}

Next, we partition the cubes of $\mathscr{W}$ which intersect $B$
into disjoint families $\mathscr{W} (L)$ labeled by $(L,B(L))\in \mathscr{Z}$
in the following way.
Let $H\in \mathscr{W}$ have nonempty intersection with $B$.
Then, either $H$ is in $\mathcal{T}$ and we set $J:= H$, or is in the domain of influence of
some $J\in \mathcal{T}$.
If $J\neq H$, then the separation between $J$ and $H$ is at most
$3\sqrt{m} \ell (J)$ and, hence, $H\subset B_{4\sqrt{m} \ell (J)} (x_J)$.
By construction there is a $B^L \in \mathscr{T}$ with $B^J\cap B^L\neq \emptyset$
and radius $s(L)\geq \frac{s(J)}{2}$. We then prescribe $H\in \mathscr{W} (L)$.
Observe that
\[
s(L)\leq 4\sqrt{m}\, \ell (L) \quad \text{and}\quad s(J) \geq \ell (J).
\]
Therefore, it also holds
\[
\ell (J) \leq 8\sqrt{m}\, \ell (L) \quad \text{and}\quad |x_J - x_L|\leq 5 s(L) \leq 20 \sqrt{m} \,\ell (L),
\]
thus implying
\[
H\subset B_{4\sqrt{m} \,\ell (J)} (x_J) \subset B_{4\sqrt{m} \ell (J) + 20 \sqrt{m}\, \ell (L)} (x_L)
\subset B_{30 \sqrt{m}\, \ell (L)} (x_L)\, .
\]

%
%
\section{Order of contact}
In this section we discuss the issues in steps (D) and (E) of the sketch of proof
in \S~\ref{ss:sketch},
i.e.~the order of contact of the normal approximation with the center manifold.

The key word for this part is \textit{frequency function}, 
which is the monotone quantity discovered by Almgren controlling the 
vanishing order of a harmonic function.
In order to explain this point, we consider first the case of a
real valued harmonic function $f:B_1\subset\R^2 \to \R$ with an expansion
in polar coordinates
\[
f(r, \theta) = a_0+\sum_{k=1}^\infty r^k\big(a_k \, \cos(k\theta) + b_k\, \sin(k\theta) \big). 
\]
How can one detect the smallest index $k$ such that $a_k$ or $b_k$ is not $0$?
It is not difficult to show that the quantity 
\begin{equation}\label{e:If}
I_f(r) := \frac{r\int_{B_r}|\nabla f|^2}{\int_{\de B_r}|f|^2}
\end{equation}
is monotone increasing in $r$ and its limit as $r\downarrow 0$ gives exactly the smallest
non-zero index in the expansion above.

$I_f$ is what Almgren calls the frequency function (and the reason for such terminology
is now apparent from the example above),
and one of the most striking discoveries of Almgren is that the monotonicity of the frequency
remains true for $Q$-valued functions and in fact allows to obtain a non-trivial blowup limit.

\medskip

In the next subsections, we see how this discussion generalizes to the case of area
minimizing currents, where an \textit{almost monotonicity} formula can be derived
for a suitable frequency defined for the $\cM$-normal approximation.

\subsection{Frequency function's estimate}
For every interval of flattening $I_j = ]s_j, t_j]$,
let $N_j$ be the normal approximation of $T_j$
on $\cM_j$.
Since the $L^2$ norm of the trace of $N_j$ may not have any connection to
the current itself (remember that $N_j$ misses a set of positive
measure of $T_j$), we need to introduce an
averaged version of the frequency function.
To this aim, consider the following piecewise linear
function $\phi:[0+\infty[ \to [0,1]$ given by
\begin{equation*}
\phi (r) :=
\begin{cases}
1 & \text{for }\, r\in [0,\textstyle{\frac{1}{2}}],\\
2-2r & \text{for }\, r\in \,\, ]\textstyle{\frac{1}{2}},1],\\
0 & \text{for }\, r\in \,\, ]1,+\infty[,
\end{cases}
\end{equation*}
and let us define a new frequency function in the following way.

\begin{definition}\label{d:frequency}
For every $r\in ]0,3]$ we define: 
\[
\bD_j (r) := \int_{\cM^j} \phi\left(\frac{d_j(p)}{r}
\right)\,|D N_j|^2(p)\, dp,
\]
and
\[
\bH_j (r) := - \int_{\cM^j} \phi'\left(\frac{d_j (p)}{r}\right)\,\frac{|N_j|^2(p)}{d(p)}\, dp\, ,
\]
where $d_j (p)$ is the geodesic distance on $\cM_j$ between $p$ and $\Phii_j (0)$.
If we have that $\bH_j (r) > 0$, then we define the {\em frequency function}
\[
\bI_j (r) := \frac{r\,\bD_j (r)}{\bH_j (r)}.
\]
\end{definition}
Note that, by the Coarea formula,
\begin{align}\label{e:Hmedia}
\bH_j (r) & = 2\,\int_{\cB_{r} \setminus \cB_{r /2} (\Phi_j (0))}  \frac{|N|^2}{d(p)} \notag\\
&= 2 \int_{r/2}^{r} \frac{1}{t} \int_{\partial \cB_t (\Phi_j (0))} |N_j|^2\, dt\, ,
\end{align}
whereas, using Fubini,
\begin{align}\label{e:Dmedia}
r\,\bD_j (r) & = \int_{\cM_j} |DN_j|^2 (x) \int_{\frac{r}{2}}^{r} {\bf 1}_{]|x|, \infty[} (t)\, dt\, d\cH^m (x)\notag\\
&= 2 \int_{\frac{r}{2}}^{r}\int_{\cB_t(\Phii_j(0))} |D{N}_j|^2\, dt.
\end{align}

This explains in which sense $\bI_j$ is an average of the quantity
introduced by F.~Almgren.

The main analytical estimate is then the following.

\begin{theorem}\label{t:frequency}
If $\eps_3$ in \eqref{e:def_R} is sufficiently small, then
there exists a constant $C>0$ (indepent of $j$) such that, if
$[a,b]\subset [\frac{s}{t}, 3]$ and $\bH_j \vert_{[a,b]} >0$, then it holds
\begin{equation}\label{e:frequency}
\bI_j (a) \leq C(1 + \bI_j (b)).
\end{equation}
\end{theorem}

To simplify the notation, we drop the index $j$ and 
omit the measure $\cH^m$ in the integrals over regions of
$\cM$.
For the proof of the theorem we need to introduce some auxiliary functions
(all absolutely continuous with respect to $r$).
We let $\de_{\hat r}$ denote the derivative along geodesics starting at $\Phii(0)$. We set
\begin{gather*}
\bE (r) := - \int_\cM \phi'\left(\textstyle{\frac{d(p)}{r}}\right)\,\sum_{i=1}^Q \langle
N_i(p), \de_{\hat r} N_i (p)\rangle\, dp\,  ,\\
\bG (r) := - \int_{\cM} \phi'\left(\textstyle{\frac{d(p)}{r}}\right)\,d(p) \left|\de_{\hat r} N (p)\right|^2\, dp,\\
\bSigma (r) :=\int_\cM \phi\left(\textstyle{\frac{d(p)}{r}}\right)\, |N|^2(p)\, dp\, .
\end{gather*}

The proof of Theorem~\ref{t:frequency} exploits
some ``integration by parts'' formulas, which in
our setting are given by the first variations for the
minimizing current.
We collect these identities in the following proposition, and proceed then
with the proof of the theorem.

\begin{proposition}\label{p:variation}
There exist dimensional constants $C,\gamma_3>0$ such that, if
the hypotheses of Theorem~\ref{t:frequency} hold and $\bI \geq 1$, then
\begin{gather}
\left|\bH' (r) - \textstyle{\frac{m-1}{r}}\, \bH (r) - \textstyle{\frac{2}{r}}\,\bE(r)\right|\leq  C \bH (r), \label{e:H'}\\
\left|\bD (r)  - r^{-1} \bE (r)\right| \leq C \bD (r)^{1+\gamma_3} + C \eps_3^2 \,\bSigma (r),\label{e:out}\\
\left| \bD'(r) - \textstyle{\frac{m-2}{r}}\, \bD(r) - \textstyle{\frac{2}{r^2}}\,\bG (r)\right|\leq
C \bD (r) + C \bD (r)^{\gamma_3} \bD' (r) + r^{-1}\bD(r)^{1+\gamma_3},\label{e:in}\\
\bSigma (r) +r\,\bSigma'(r) \leq C  \, r^2\, \bD (r)\, \leq C r^{2+m} \eps_3^{2}.\label{e:Sigma1}
\end{gather}
\end{proposition}

We assume for the moment the proposition and prove the theorem.

\begin{proof}[Proof of Theorem \ref{t:frequency}.]
It enough to consider the case in which $\bI >1$ on $]a,b[$.
Set $\bOmega(r) := \log \bI(r)$.
By Proposition~\ref{p:variation}, if $\eps_3$ is sufficiently small, then 
\begin{equation}\label{e:out2}
\frac{\bD (r)}{2} 
\leq
\frac{\bE (r)}{r} 
\leq
2\, \bD (r),
\end{equation}
from which we conclude that $\bE>0$ over the interval $]a, b'[$.
Set for simplicity $\bF(r) := \bD(r)^{-1} - r \bE(r)^{-1}$, and compute
\[
- \bOmega'  (r) = \frac{\bH' (r)}{\bH(r)} - \frac{\bD'(r)}{\bD(r)} - \frac{1}{r} 
\stackrel{\eqref{e:out}}{=} \frac{\bH'(r)}{\bH(r)} - \frac{r\bD'(r)}{\bE(r)} - \bD'(r) \bF (r) - \frac{1}{r}.
\]
Again by Proposition~\ref{p:variation}:
\begin{equation}\label{e:pezzo_1}
\frac{\bH'(r)}{\bH(r)} \stackrel{\eqref{e:H'}}{\leq} \frac{m-1}{r} + C + \frac2r\,\frac{\bE(r)}{\bH(r)},
\end{equation}
\begin{equation}\label{e:E denominatore}
|\bF(r)| \stackrel{\eqref{e:out}}{\leq} C \, \frac{r (\bD (r)^{1+\gamma_3} + \bSigma (r))}{\bD (r)\,\bE (r)} \stackrel{\eqref{e:out2}}{\leq} C \, \bD (r)^{\gamma_3-1}+ C\,\frac{\bSigma (r)}{\bD (r)^2},
\end{equation}
\begin{align}\label{e:pezzo_2}
- \frac{r\bD'(r)}{\bE(r)} &\stackrel{\eqref{e:in}}{\leq} \left(C - \frac{m-2}{r}\right) \frac{r\bD(r)}{\bE(r)} - \frac2r \, \frac{\bG(r)}{\bE(r)}\notag\\
&\quad+  C\frac{r \bD(r)^{\gamma_3} \bD' (r) + \bD (r)^{1+\gamma_3}}{\bE(r)}\nonumber\\
&\leq C - \frac{m-2}{r} 
+ \frac{C}{r} \bD(r) |\bF(r)|
- \frac2r \, \frac{\bG(r)}{\bE(r)}\notag\\
& \quad +  C \bD(r)^{\gamma_3-1} \bD' (r) + C\frac{\bD (r)^{\gamma_3}}{r}\notag\\
& \stackrel{\eqref{e:Sigma1},\, \eqref{e:E denominatore}}{\leq}C - \frac{m-2}{r} 
- \frac2r \, \frac{\bG(r)}{\bE(r)}
+  C \bD(r)^{\gamma_3-1} \bD' (r) + C\,r^{\gamma_3\,m -1},
\end{align}
where we used the rough estimate 
$\bD(r) \leq C\, r^{m+2-2\delta_2}$ coming from \eqref{e:Dir_regional} of
Theorem~\ref{t:approxN} and the condition (Stop).

By Cauchy-Schwartz, we have
\begin{equation}\label{e:cauchy-schwartz}
\frac{\bE(r)}{r\bH(r)}\leq \frac{\bG(r)}{r \bE(r)}.
\end{equation}
Thus, by \eqref{e:pezzo_1},
\eqref{e:pezzo_2} and  \eqref{e:cauchy-schwartz}, we conclude
\begin{align}
- \bOmega' (r) &\leq C +C\,r^{\gamma_3\,m -1}+
C r \bD(r)^{\gamma_3-1} \bD' (r) 
- \bD'(r) \bF(r)\notag\\
& \stackrel{\eqref{e:E denominatore}}{\leq}C\,r^{\gamma_3\,m -1}+
C\bD(r)^{\gamma_3-1} \bD'(r) + C \frac{\bSigma (r) \bD' (r)}{\bD(r)^2}.\label{e:grosso} 
\end{align}
Integrating \eqref{e:grosso} we conclude:
\begin{multline*}
\bOmega (a) - \bOmega (b)\leq C + C \left(\bD(b)^{\gamma_3} - \bD(a)^{\gamma_3}\right) \notag\\
+ C \left[ \frac{\bSigma (a)}{\bD (a)} - \frac{\bSigma (b)}{\bD (b)} + \int_a^b \frac{\bSigma' (r)}{\bD (r)}\, dr\right]\stackrel{\eqref{e:Sigma1}}{\leq} C.\notag
\end{multline*}
\end{proof}

\subsubsection{Proof of Proposition~\ref{p:variation}}

The remaining part of this subsection is devoted to give some arguments for the proof
of the first variation formulas.

The estimate \eqref{e:H'} follows from a straightforward computation:
using the area formula and setting $y =rz$, we have
\begin{align*}
\bH (r) & 
= - r^{m-1}\int_{T_q\cM} \frac{\phi'\left(|z|\right)}{|z|}\,|N|^2(\exp (r z))\, \textbf{J}\exp (r z)\, dx,
\end{align*}
and differentiating under the integral sign, we easily get \eqref{e:H'}:
\begin{align*}
\bH' (r) = {} & - (m-1)\,r^{m-2}\int_{T_q\cM} \frac{\phi'(|z|)}{|z|}\,|N|^2(\exp (r z))\,\textbf{J}\exp (r z)\, dz
\notag\\
& - 2 \, r^{m-1}\int_{T_q\cM} \phi'(|z|)\,\sum_i \left\langle N_i , \de_{\hat{r}} N_i \right\rangle(\exp (r z))\,\textbf{J}\exp (r z)\, dz
\notag\\
& - r^{m-1}\int_{T_q \cM} \frac{\phi'(|z|)}{|z|}\,|N|^2(\exp (r z))\,\frac{d}{d r} \textbf{J}\exp (r z)\, dz
\notag\\
= {} & \frac{m-1}{r}\, \bH (r) + \frac{2}{r}\, \bE (r) +O(1)\, \bH (r),
\end{align*}
where we the following simple fact for the Jacobian of the exponential map
$\frac{d}{d r} \textbf{J}\exp (r\,z) = O(1)$, because $\cM$ is a $C^{3,\kappa}$ submanifold
and the exponential map $\textup{exp}$ is a $C^{2, \kappa}$ map.

Similarly, \eqref{e:Sigma1} follows by simple computation which involve a Poincar\'e
inequality: namely, if $\bI \geq 1$, then
\begin{equation}\label{e:L2_pieno2}
\int_{\mathcal{B}_r (q)} |N|^2 \leq C\,r^2\bD (r).
\end{equation}
We refer to \cite{DS5} for the details of the proof.

\medskip

Here we try to explain the remaining two estimates, which instead are
connected to the first variation $\delta T (X)$ of the area minimizing current $T$ along
a vector field $X$.

The idea is the following: since the first variations of $T$ are zero,
we compute them using its
approximation $N$ and derive the integral equality in the Proposition~\ref{p:variation}.
To understand the meaning of these estimates, consider
$u:\R^m \to \R^n$ a harmonic function. Then,
computing the variations of the Dirichlet energy
of $u$ leads to the following two identities:
\begin{gather*}
\int_{B_r} |Du|^2 = \int_{\de B_r} u \cdot \frac{\de u}{\de \nu},\\
\int_{\de B_r} |Du|^2 = \frac{m-2}{r} \int_{B_r} |Du|^2 + 2 \int_{\de B_r} \left\vert \frac{\de u}{\de \nu} \right\vert^2,
\end{gather*}
which are the exact analog of \eqref{e:out} and \eqref{e:in} without any error term.
What we need to do is then to replace the Dirichlet
energy with the area functional, and to consider 
the fact that the normal approximation $N$ is only approximately
stationary with respect to this functional.

\medskip

We start fixing a tubular neighborhood $\bU$ of $\cM$ and the normal
projection $\p:\bU\to \cM$. Observe that
$\p\in C^{2,\kappa}$.
We will consider:
\begin{itemize}
\item[(1)] the {\em outer variations}, where $X (p)= X_o (p) := \phi \left(\frac{d(\p(p))}{r}\right) \, (p - \p(p) )$.
\item[(2)] the {\em inner variations}, where $X (p) = X_i (p):= Y(\p(p))$ with
\[
Y(p) := \frac{d(p)}{r}\,\phi\left(\frac{d(p)}{r}\right)\, \frac{\de}{\de \hat r} \quad
\forall \; p \in \cM.
\]
\end{itemize}

Consider now the map $F(p) := \sum_i \a{p+ N_i (p)}$ and the current $\bT_F$
associated to its image.
Observe that $X_i$ and $X_o$ are supported in $\p^{-1} (\cB_r (q))$
but none of them is {\em compactly} supported. 
However, it is simple to see that
$\delta T(X) = 0$.
Then, we have 
\begin{multline}\label{e:Err4-5}
|\delta \bT_F (X)| = |\delta \bT_F (X) - \delta T (X)|\\
\leq \underbrace{\int_{\supp (T)\setminus \im (F)}  \left|\dv_{\vec T} X\right|\, d\|T\|
+ \int_{\im (F)\setminus \supp (T)} \left|\dv_{\vec \bT_F} X\right|\, d\|\bT_F\|}_{{\rm Err}_4},
\end{multline}
where $\im(F)$ is the image of the map $F(x) = \sum_i\a{(x,N_i(x))}$, i.e.~the
support of the current $\bT_F$.

Set now for simplicity $\varphi_r (p) :=  \phi \big(\frac{d(p)}{r}\big)$.
It is not hard to realize that
the mass of the current $\bT_F$ can be expressed in the following way:
\begin{align}
\mass (\mathbf{T}_F) = {}&  Q \, \cH^m (\cM) - Q\int_\cM \langle H, \etaa\circ N\rangle
+ \frac{1}{2} \int_\cM |D N|^2\nonumber\\
&+ \int_\cM \sum_i \Big(P_2 (x,N_i) +  P_3 (x, N_i, DN_i) + R_4 (x, DN_i)\Big),  \label{e:taylor_area}
\end{align}
where $P_2$, $P_3$ and $R_4$ are quadratic, cubic and fourth order errors terms
(see \cite[Theorem 3.2]{DS2})
One can then compute the first variation of a push-forward current $\bT_F$ and obtain (cp.~\cite[Theorem 4.2]{DS2})
\begin{equation}\label{e:ov graph}
\delta \bT_F (X_o) = \int_\cM\Big( \ph_r\,|D N|^2 + \sum_{i=1}^Q N_i\otimes \nabla \ph_r : D N_i \Big) + \sum_{j=1}^3 \textup{Err}^o_j,
\end{equation}
where the errors ${\rm Err}^o_j$ satisfy
\begin{gather}
{\rm Err}_1^o = - Q \int_\cM \varphi_r \langle H_\cM, \etaa\circ N\rangle,\label{e:outer_resto_1}\\
|{\rm Err}_2^o| \leq C \int_\cM |\varphi_r| |A|^2|N|^2,\label{e:outer_resto_2}\\
|{\rm Err}_3^o| \leq C \int_\cM \big(|N| |A| + |DN|^2 \big) \big( |\varphi_r| |DN|^2  + |D\varphi_r| |DN| |N| \big)\label{e:outer_resto_3}\,,
\end{gather}
here $H_\cM$ is the mean curvature vector of $\cM$.
Plugging \eqref{e:ov graph} into \eqref{e:Err4-5}, we then conclude
\begin{equation}\label{e:ov_con_errori}
\left| \bD (r) - r^{-1} \bE (r)\right| \leq \sum_{j=1}^4 \left|\textup{Err}^o_j\right|\, ,
\end{equation}
where ${\rm Err}_4^o$ corresponds
to ${\rm Err}_4$ of \eqref{e:Err4-5} when $X=X_o$.
Arguing similarly with $X=X_i$ (cp.~\cite[Theorem 4.3]{DS2}),
we get
\begin{align}\label{e:iv graph}
\delta \bT_F (X_i) =& \frac{1}{2} \int_\cM\Big(|D N|^2 {\rm div}_{\cM} Y  - 2
\sum_{i=1}^Q  \langle D N_i : ( D N_i \cdot D_{\cM} Y)\rangle \Big) + \sum_{j=1}^3 \textup{Err}^i_j\, ,
\end{align}
where this time the errors ${\rm Err}^i_j$  satisfy
\begin{gather}
{\rm Err}_1^i = - Q \int_{ \cM}\big( \langle H_\cM, \etaa \circ N\rangle\, {\rm div}_{\cM} Y + \langle D_Y H_\cM, \etaa\circ N\rangle\big)\, ,\label{e:inner_resto_1}\\
|{\rm Err}_2^i| \leq C \int_\cM |A|^2 \left(|DY| |N|^2  +|Y| |N|\, |DN|\right), \label{e:inner_resto_2}
\end{gather}
\begin{align}
|{\rm Err}_3^i|&\leq C \int_\cM |Y| |A| |DN|^2 \big(|N| + |DN|\big) \notag\\
&\quad+ |DY| \big(|A|\,|N|^2 |DN| + |DN|^4\big)\Big)\label{e:inner_resto_3}\, .
\end{align}
Straightforward computations lead to 
\begin{align}\label{e:DY}
D_{\cM} Y (p) = \phi'\left(\frac{d(p)}{r}\right) \, \frac{d(p)}{r^2} \, \frac{\de}{\de \hat r} \otimes \frac{\de}{\de \hat r} + \phi \left(\frac{d(p)}{r}\right) \left( \frac{\Id}{r} + O(1)\right)\, ,
\end{align}
\begin{align}\label{e:divY}
\dv_\cM\, Y (p) 
& = \phi'\left(\frac{d(p)}{r}\right) \, \frac{d(p)}{r^2} + \phi\left(\frac{d(p)}{r}\right) \, \left(\frac{m}{r} + O(1) \right)\, .
\end{align}
Plugging \eqref{e:DY} and \eqref{e:divY} into \eqref{e:iv graph} and using \eqref{e:Err4-5} we then
conclude
\begin{equation}\label{e:iv_con_errori}
\left| \bD' (r) - (m-2)r^{-1} \bD (r) - 2 r^{-2} \bG (r)\right|
\leq C \bD (r) + \sum_{j=1}^4 \left|{\rm Err}^i_j\right|\, .
\end{equation}
Proposition~\ref{p:variation} is then proved by the estimates of the errors
terms done in the next subsection.

\subsubsection{Estimates of the errors terms}
We consider the family of pairs $\mathscr{Z} =\{(J_i, B(J_i))\}_{i\in \N}$
introduced in the previous section,
and set
\[
\mathcal{B}^i := \Phii (B (J_i))\quad
\text{and}
\quad \mathcal{U}_i = \cup_{H\in \mathscr{W} (J_i)} \Phii (H) \cap \cB_r (q)\, .
\] 
Set $\mathcal{V}_i:= \mathcal{U}_i\setminus \mathcal{K}$, where $\mathcal{K}$ is the coincidence set of Theorem~\ref{t:approxN}.
By a simple application of Theorem~\ref{t:approxN} we derive the following estimates:
\begin{gather}
\int_{\mathcal{U}_i} |\etaa\circ N| \leq C {E}\,\ell_i^{2+m+\frac{\gamma_2}{2}} + C \int_{\cU^i} |N|^{2+\gamma_2},\label{e:media}\\
\int_{\mathcal{U}_i} |DN|^2 \leq C {E}\, \ell_i^{m+2-2\delta_2},\label{e:Dirichlet_sopra}\\
\|N\|_{C^0 (\mathcal{U}_i)} + \sup_{p\in \supp (T) \cap \p^{-1} (\mathcal{U}_i)} |p- \p (p)| \leq C {E}^{\frac{1}{2m}} \ell_i^{1+\beta_2},\label{e:N_sopra}\\
\Lip (N|_{\mathcal{U}_i}) \leq C {E}^{\gamma_2} \ell_i^{\gamma_2},\label{e:Lipschitz}\\
\mass (T\res \p^{-1} (\mathcal{V}_i)) + \mass (\bT_F \res \p^{-1} (\mathcal{V}_i)) \leq C {E}^{1+\gamma_2} \ell_i^{m+2+\gamma_2}.\label{e:errori_massa}
\end{gather}

Observe that the separation between $\cB^i$ and
$\partial \mathcal{B}_r (q)$ is larger than $\ell(J_i)/4$ by Proposition~\ref{p:covering} (i), and then
$\varphi_r (p) = \phi \big(\frac{d(p)}{r}\big)$ satisfies
\begin{align}
\inf_{p\in \cB^i} \varphi_r (p) \geq (4r)^{-1} \ell_i\label{e:peso_1}\, ,
\end{align}
where $\ell_i := \ell (J_i)$. From this and Proposition~\ref{p:covering} (iii), 
we also obtain
\[
\sup_{p\in \mathcal{U}_i} \varphi_r (p)
- \inf_{p\in \mathcal{U}_i}\varphi_r (p) \leq C \Lip (\varphi_r) \ell_i \leq \frac{C}{r}\,\ell_i \stackrel{\eqref{e:peso_1}}{\leq}
C \inf_{p\in \mathcal{B}_i} \varphi_r (p)\, ,
\]
which translates into
\begin{equation}\label{e:peso_2}
\sup_{p\in \mathcal{U}_i} \varphi_r (p) \leq C \inf_{p\in \cB^i} \varphi_r (p)\, .
\end{equation}
Moreover, by an application of the \textit{splitting-before-tilting} estimates
in Proposition~\ref{p:separ} and Proposition~\ref{p:splitting}, we infer that 
\begin{align}
\int_{\cB^i} |N|^2 &\geq c \, {E}^{\frac{1}{m}} \ell_i^{m+2+2\beta_2}  \quad 
\mbox{if $L_i\in \mathscr{W}_h$},\label{e:N_sotto}\\
\int_{\cB^i} |DN|^2 &\geq c \, {E}\, \ell_i^{m+2-2\delta_2}
\quad\mbox{if $L_i \in \mathscr{W}_e$}\, .\label{e:Dirichlet_sotto}
\end{align}

This easily implies the following estimates under the hypotheses $\bI\geq 1$:
by applying \eqref{e:L2_pieno2}, \eqref{e:peso_1}, \eqref{e:N_sotto} and
\eqref{e:Dirichlet_sotto}, we get, for suitably chosen $\gamma(t), C(t)>0$,
\begin{align}\label{e:gamma(a)}
\sup_i{E}^t\Big[\ell_i^t + \Big(\inf_{\cB^i} \varphi_r\Big)^{\frac{t}{2}} \ell_i^{\frac{t}{2}}\Big]
&\leq C(t) \sup_i\Big( \int_{\cB^i} \varphi_r (|DN|^2 + |N|^2)\Big)^{\gamma (t)} \notag\\
&\leq C(t) \bD (r)^{\gamma(t)}\,,
\end{align}
and similarly
\begin{align}\label{e:D_globale}
\sum_i \big(\inf_{\cB^i} \varphi_r \big) {E}\, \ell_i^{m+2+\frac{\gamma_2}{4}} 
&\leq C \sum_i \int_{\cB^i} \varphi_r (|DN|^2 +|N|^2) \notag\\
&\leq  C \bD (r)\,,
\end{align}
\begin{align}\label{e:D'_globale}
\sum_i {E}\, \ell_i^{m+2+\frac{\gamma_2}{4}} &\leq C \int_{\cB_r (q)} \big(|DN|^2+|N|^2\big)\notag\\
&\leq C \big(\bD (r) + r \bD' (r)\big).
\end{align}

\medskip

We can now pass to
estimate the errors terms in \eqref{e:out} and \eqref{e:in} in order to conclude
the proof of Proposition~\ref{p:variation}.

\medskip

\noindent {\bf Errors of type 1.}
By Theorem~\ref{t:cm}, the map $\phii$ defining the center
manifold satisfies $\|D\phii\|_{C^{2, \kappa}} \leq C\, {E}^{\frac12}$,
which in turn implies $\|H_\cM\|_{L^\infty} + \|DH_\cM\|_{L^\infty} \leq C\, {E}^{\frac12}$ (recall that $H_\cM$ denotes
the mean curvature of $\cM$).
Therefore, by \eqref{e:peso_2}, \eqref{e:media}, \eqref{e:D_globale} and \eqref{e:gamma(a)}, we get
\begin{align}
\left|{\rm Err}^o_1\right| &\leq C \int_\cM \varphi_r \,|H_\cM|\,
|\etaa\circ N| \notag\allowdisplaybreaks\\
& \leq C \, {E}^{\frac12} \sum_j  
\Big(\big(\sup_{\mathcal{U}_i} \varphi_r\big)\,{E}\, \ell_j^{2+m+\gamma_2} + C \int_{\mathcal{U}_j} \ph_r|N|^{2+\gamma_2}\Big)\notag\allowdisplaybreaks\\
&\leq C \bD (r)^{1+\gamma_3} 
+ C \sum_j {E}^{\frac12} \,\ell_j^{\gamma_2(1+\beta_2)} \int_{\mathcal{U}_j} \ph_r|N|^2
\leq C \bD (r)^{1+\gamma_3}\,,\notag\allowdisplaybreaks
\end{align}
and analogously
\begin{align}
\left|{\rm Err}^i_1\right| &\leq C\,r^{-1} \int_\cM \big(|H_\cM|+ |D_Y H_\cM|\big)
|\etaa\circ N| \notag\allowdisplaybreaks\\
& \leq C \,r^{-1} {E}^{\frac12} \sum_j  
\Big(\,{E}\, \ell_j^{2+m+\gamma_2} + C \int_{\mathcal{U}_j} |N|^{2+\gamma_2}\Big)\notag\\
&\leq C\,r^{-1} \bD (r)^{\gamma}\big(\bD(r)+r\,\bD'(r)\big).\notag
\end{align}

\noindent {\bf Errors of type 2.}
From $\|A\|_{C^0} \leq C \|D \phii\|_{C^2} \leq C {E}^{\frac{1}{2}} \leq C \eps_3$, it follows that
${\rm Err}^o_2 \leq C \eps_3^2 \bSigma (r)$.
Moreover, since $|D X_i| \leq Cr^{-1}$, \eqref{e:L2_pieno2} leads to
\begin{align}
\left|{\rm Err}^i_2\right| &\leq C r^{-1} \int_{\cB_r (p_0)} |N|^2 + C \int \varphi_r |N| |DN|
\leq C \bD (r)\, .\notag
\end{align}

\noindent {\bf Errors of type 3.}
Clearly, we have
\begin{align}
\left|{\rm Err}^o_3\right| \leq & \underbrace{\int \varphi_r \left(|DN|^2 |N| + |DN|^4\right)}_{I_1}
+ C\underbrace{r^{-1} \int_{\cB_r (q)} |DN|^3 |N|}_{I_2}\notag\\
& + C\underbrace{r^{-1} \int_{\cB_r (q)} |DN| |N|^2}_{I_3}\, .\notag
\end{align}
We estimate separately the three terms (recall that
$\gamma_2> 4\delta_2$):
\begin{align}
I_1 & \leq \int_{\cB_r (p_0)} \varphi_r (|N|^2 |DN| + |DN|^3)
\leq I_3 +  C \sum_j \sup_{\cU_j} \varphi_r {E}^{1+\gamma_2} \ell_j^{m+2+\frac{\gamma_2}{2}}\nonumber\\
&\stackrel{\eqref{e:D_globale}\,\&\,\eqref{e:gamma(a)}}{\leq} I_3 + C \bD (r)^{1+\gamma_3},\notag
\end{align}
\begin{align}
I_2 &\leq C r^{-1} \sum_j {E}^{1+\frac{1}{2m}+\gamma_2} \ell_j^{m+3+\beta_2 +\frac{\gamma_2}{2}}\notag\\
&\stackrel{\eqref{e:peso_2}}{\leq} C \sum_j {E}^{1+\frac{1}{2m}+\gamma_2} \ell_j^{m+2+\beta_2 +\frac{\gamma_2}{2}} \inf_{\cB^j} \varphi_r
\stackrel{\eqref{e:D_globale}\,\&\,\eqref{e:gamma(a)}}{\leq} C \bD (r)^{1+\gamma_3},\notag
\end{align}
\begin{align}
I_3 &\leq C r^{-1} \sum_j {E}^{\gamma_2} \ell_j^{\gamma_2} \int_{\mathcal{U}_j} |N|^2
\stackrel{\eqref{e:gamma(a)}}{\leq} C r^{-1} \bD (r)^{\gamma_3}  \int_{\cB_r (q)} |N|^2
\stackrel{\eqref{e:L2_pieno2}}{\leq} C \bD (r)^{1+\gamma_3}\notag
\end{align}
For what concerns the inner variations, we have 
\begin{align}
|{\rm Err}^i_3| &\leq C \int_{\cB_r (q)} \big(r^{-1}|DN|^3 + r^{-1} |DN|^2 |N| + r^{-1} |DN| |N|^2\big)\, .\notag
\end{align}
The last integrand corresponds to $I_3$, while the remaining part can be estimated as
follows:
\begin{align}
\int_{\cB_r (q)} r^{-1}(|DN|^3 + |DN|^2 |N|)&\leq C \sum_j r^{-1}({E}^{\gamma_2} \ell_j^{\gamma_2} +
{E}^{\frac{1}{2m}} \ell_j^{1+\beta_2}) \int_{\mathcal{U}_j} |DN|^2\notag\allowdisplaybreaks\\
&\stackrel{\mathclap{\eqref{e:gamma(a)}}}{\leq} \, C\,r^{-1} \bD (r)^{\gamma_3} \int_{\cB_r (q)} |DN|^2\notag\allowdisplaybreaks\\
&\leq 
C \bD (r)^{\gamma_3} \left(\bD' (r) + r^{-1}\bD (r)\right)\, .\notag
\end{align}

\medskip

\noindent {\bf Errors of type 4.} We compute explicitly
\begin{align}
|D X_o (p)| &\leq 2\, |p - \p(p)| \, \frac{|Dd(\p(p),q)|}{r} + \varphi_r (p) \, |D(p-\p(p))|\notag\\
&\leq C\,\left(\frac{|p - \p(p)|}{r} + \varphi_r (p)\right)\, .\notag
\end{align}
It follows readily from \eqref{e:Err4-5}, \eqref{e:N_sopra}
and \eqref{e:errori_massa} that
\begin{align}
|{\rm Err}^o_4| &\leq \sum_i C \Big(r^{-1} {E}^{\frac{1}{2m}} \ell_i^{1+\beta_2}
+ \sup_{\mathcal{U}_i} \varphi_r\Big) {E}^{1+\gamma_2} \ell_i^{m+2+\gamma_2}\notag\allowdisplaybreaks\\
&\stackrel{\eqref{e:peso_1}\,\& \,\eqref{e:peso_2}}{\leq} C \sum_i
\left[
{E}^{\gamma_2} \ell_i^{\frac{\gamma_2}{4}}\right] 
\inf_{\cB_i} \varphi_r\, {E}\, \ell_i^{m+2+\frac{\gamma_2}{4}}
\stackrel{\eqref{e:D_globale}\,\&\eqref{e:gamma(a)}}{\leq} C \bD (r)^{1+\gamma_3}.
\label{e:stima4-5o}
\end{align}
Similarly, since $|D X_i|\leq C r^{-1}$, we get
\begin{align*}
{\rm Err}^i_4 &\leq C r^{-1} \sum_j \Big({E}^{\gamma_2} \ell_j^{\frac{\gamma_2}{2}}\Big)
{E}\,\ell_j^{m+2+\frac{\gamma_2}{2}}\\
&\stackrel{\eqref{e:D'_globale}\,\&\,\eqref{e:gamma(a)}}{\leq} C \bD (r)^\gamma \left(\bD' (r) +r^{-1} \bD(r)\right)\, .\notag
\end{align*}

\begin{remark}
Note that the ``superlinear'' character of the estimates in Theorem~\ref{t:approxN}
has played a fundamental role in the control of the errors.
\end{remark}

\subsection{Boundness of the frequency}
We have proven in the previous subsection that the frequency of the $\cM$-normal approximation remains bounded within a center manifold in the corresponding interval of flattening. 
In order to pass into the limit along the different center manifolds,
we need also to
show that the frequency remains bounded in passing from one to the other.
This is again a consequence of the \textit{splitting-before-tilting} estimates and we provide here some details of the proof, referring to \cite{DS4}
for the complete argument.

To simplify the notation, we set $p_j := \Phii_j (0)$ and write simply $\cB_\rho$
in place of $\cB_\rho (p_j)$ .

\begin{theorem}[Boundedness of the frequency functions]\label{t:boundedness}
If the intervals of flattening are infinitely many, then there is a number $j_0\in \mathbb N$ such that
\begin{equation}\label{e:finita2}
\bH_j>0 \mbox{ on $]\frac{s_j}{t_j}, 3[$ for all $j\geq j_0$} \quad \mbox{and} \quad 
\sup_{j\geq j_0} \sup_{r\in ]\frac{s_j}{t_j}, 3[} \bI_j (r) <\infty\, .
\end{equation}
\end{theorem}

\begin{proof}[Sketch of the proof]
We partition the extrema
$t_j$ of the intervals of  flattening into two different classes:
\begin{itemize}
\item[$(A)$] those such that $t_j = s_{j-1}$;
\item[$(B)$] those such that $t_j<s_{j-1}$.
\end{itemize}
If $t_j$ belongs to $(A)$, set $r:= \frac{s_{j-1}}{t_{j-1}}$. Let $L \in \sW^{(j-1)}$ be
a cube of the Whitney decomposition such that $c_s\, r \leq \ell(L)$ and 
$L\cap \bar B_r(0, \pi) \neq \emptyset$.
Since this cube of the Whitney decomposition at step $j-1$ has
size comparable with the distance to the origin, and the next center manifold starts
at a comparable radius, the splitting property of the normal approximation needs to
hold also for the new approximation: namely, one can show that
there exists a constant $\bar c_s >0$ such that
\[
\int_{\B_2\cap \cM_j} |N_j|^2 
\geq \bar{c}_s E^j:=\bE (T_j, \B_{6\sqrt{m}}),
\]
which obviously gives $\bH_{N_j} (3) \geq c E^j$, and than
$\bI_{N_j} (3)$ is smaller than a given constant, independent of $j$, thus
proving the theorem.

In the case $t_j$ belongs to the class $(B)$, 
then, by construction there is $\eta_j \in ]0,1[$ such that
$\bE ((\iota_{0, t_j})_\sharp T, \B_{6\sqrt{m}(1+\eta_j)}) > \eps_3^2$.
Up to extracting a subsequence, we can assume that $(\iota_{0, t_j})_\sharp T$
converges to a cone $S$: the convergence is strong enough to conclude that
the excess
of the cone is the limit of the excesses of the sequence. Moreover
(since $S$ is a cone), the excess $\bE (S, \B_r)$ is independent of $r$.
We then conclude 
\[
\eps_3^2 \leq  \liminf_{j\to\infty, j\in (B)} \bE (T_j, \B_3)\, .
\]
Thus, it follows again from the splitting phenomenon
(see for details \cite[Lemma~5.2]{DS5}) that
$\liminf_{j\to\infty, j\in (B)} \bH_{N_j} (3) > 0$. Since $\bD_{N_j}(3) \leq C E^j\leq C\eps_3^2$, we achieve that $\limsup_{j\to\infty, j\in (B)} \bI_{N_j} (3)>0$, and conclude as before.
\end{proof}

%
%
\section{Final blowup argument}

We are now ready for the conclusion of the blowup argument, i.e.~for
the discussion of steps (F) and (G) of \S~\ref{ss:sketch}.

To this aim we recall here the main results obtained so far.

We start with an $m$-dimensional area minimizing integer
rectifiable $T$ in $\R^{m+n}$ with $\de T = 0$ and $0\in \rD_Q (T)$,
such that there exists a sequence of radii $r_k\downarrow 0$
satisfying
\begin{gather}
\lim_{k\to+\infty}\bE(T_{0,r_k}, \B_{10}) = 0,\label{e:seq1bis}\\
\lim_{k\to+\infty} \cH^{m-2+\alpha}_\infty (\rD_Q (T_{0,r_k}) \cap \B_1) > \eta>0,\label{e:seq2bis}\\
\cH^m \big((\B_1\cap \supp (T_{0, r_k}))\setminus \rD_Q (T_{0,r_k})\big) > 0\quad \forall\; k\in\N,\label{e:seq3bis}
\end{gather}
for some constant $\alpha, \eta>0$.
In the process of solving the centering problem for such currents we have
obtained the following:
\begin{enumerate}
\item the intervals of flattening $I_j = ]s_j, t_j]$,
\item the center manifolds $\cM_j$,
\item the $\cM_j$-normal approximations $N_j : \cM_j \to \Iq(\R^{m+n})$,
\end{enumerate}
satisfying the conclusions of Theorem~\ref{t:cm} and Theorem~\ref{t:approxN}.
It follows from the very definition of intervals of flattening that each
$r_k$ has to belong to one of these intervals.
Therefore, in order to fix the ideas and to simplify the notation,
we will in the sequel assume that there are infinitely many intervals
of flattening and that $r_k \in I_k$: note that this is not a serious
restriction, and everything holds true also in the case of finitely many
intervals of flattening.

By the analysis of the order of contact and the estimate on the
frequency function, see Theorem~\ref{t:frequency} and Theorem~\ref{t:boundedness},
we have also derived the information
\begin{equation}\label{e:Ifinale}
\sup_{j\in \N} \sup_{r \in \left]\frac{s_j}{t_j},3\right]} \bI_j(r) <+\infty.
\end{equation}
The ultimate consequence of this estimate, thus
clarifying the discussion about the non-triviality of the blowup
process, is the following proposition.

\begin{proposition}[Reverse Sobolev]\label{c:rev_Sob}
There exists a constant $C>0$ with this property: for every $j\in \N$,
there exists $\theta_j\in \left]\frac{3\,r_j}{2\,t_j}, 3\frac{r_j}{t_j}\right[$
such that
\begin{equation}\label{e:rev_Sob}
\int_{\cB_{\theta_j}(\Phii_j(0))} |D{N}_j|^2 \leq C\left(\frac{t_j}{r_j}\right)^2
\int_{\cB_{\theta_j} (\Phii_j(0))} |{N}_j|^2\, .  
\end{equation}
\end{proposition}

\begin{proof} 
Set for simplicity $r:= \frac{r_j}{t_j}$ and drop the subscript $_j$ in the sequel.
Using \eqref{e:Hmedia}, \eqref{e:Dmedia}
and \eqref{e:Ifinale}, there exists $C>0$ such that
\begin{align*}
\int_{\frac32 r}^{3r} dt \int_{\cB_t(\Phii(0))} |D{N}|^2 &= \frac32 r\,\bD (3r) \leq
C \,\bH (3r)\\
&= C \int_{\frac32 r}^{3r} dt
\frac{1}{t} \int_{\partial \cB_t(\Phii(0))} |{N}|^2\, .  
\end{align*}
Therefore, there must be $\theta\in [\frac{3}{2}r, 3r]$ satisfying
\begin{equation}\label{e:un_raggio_buono}
\int_{\cB_{\theta}(\Phii(0))} |D{N}|^2 \leq \frac{C}{\theta} \int_{\partial \cB_{\theta}(\Phii(0))} |{N}|^2\, .
\end{equation}
This is almost the desired estimate.
In oder to replace the boundary integral with a bulk integral in the
right hand side of \eqref{e:un_raggio_buono}, we argue by integrating along radii
in a similar way to the case of single valued functions.
Fix indeed any $\sigma \in ]\theta/2, \theta[$ and any point $x\in \partial \cB_\theta (\Phi(0))$. Consider the geodesic line $\gamma$
passing through $x$ and $\Phi (0)$, and let $\hat{\gamma}$ be the arc on $\gamma$ having one endpoint $\bar{x}$ in $\partial \cB_\sigma (\Phi (0))$ and one endpoint equal to $x$. Using \cite[Proposition 2.1(b)]{DS1} and the fundamental theorem of calculus, we easily conclude
\[
|N (x)| \leq |N (\bar x)| + \int_{\hat{\gamma}} |DN| |N|\, .
\]
Integrating this inequality in $x$ and recalling that $\sigma > s/2$ we then easily conclude
\[
\int_{\partial \cB_\theta (\Phi (0))} |N|^2 \leq C \int_{\partial \cB_\sigma (\Phi (0))} |N|^2 + C \int_{\cB_\theta (\Phi (0))} |N|\,|DN|\, .
\]
We further integrate in $\sigma$ between $s/2$ and $s$ to achieve
\begin{align}\label{e:young}
\theta \int_{\partial \cB_{\theta}(\Phii(0))} |{N}|^2 & \leq C
 \int_{\cB_{\theta}(\Phii(0))} \left(|{N}|^2 + \theta\,|{N}|\, |D{N}|\right)\notag\\
&\leq \frac{\theta^2}{2\,C} \int_{\cB_{\theta}(\Phii(0))} |D{N}|^2 + C \int_{\cB_{\theta}(\Phii(0))} |{N}|^2\, .
\end{align}
Combining \eqref{e:young} with  \eqref{e:un_raggio_buono}
we easily conclude \eqref{e:rev_Sob}.
\end{proof}

\subsection{Convergence to a Dir-minimizer}
We can now define the final blowup sequence, because the Reverse Sobolev
inequality proven in Proposition~\ref{c:rev_Sob} gives the right radius $\theta_k$ for
assuring compactness of the corresponding maps.
To this aim set $\bar{r}_k := \frac{2}{3} \theta_k\,t_{k} \in [r_k, 2\,r_k]$,
and rescale the current and the maps accordingly:
\[
\bar T_k := (\iota_{0,\bar r_k})_\sharp T\quad \text{and}\quad
\bar\cM_k := \iota_{0,\bar r_k/t_k} \cM_k,
\]
and $\bar{N}_k : \bar\cM_k \to \R^{m+n}$ for the rescaled $\bar\cM_k$-normal approximations
given by 
\[
\bar{N}_k (p) := \frac{t_k}{\bar{r}_k} N_{k} \left(\frac{\bar{r}_k p}{t_k}\right).
\]
Note that the ball $\cB_{s_k} \subset \cM_k$
is sent into the ball $\cB_{\frac{3}{2}} \subset \bar\cM_k$.
Moreover, via some elementary regularity theory of
area minimizing currents, one deduces that
\begin{enumerate}
\item $\bE(\bar T_k, \B_{\frac12}) \leq C\bE(T, \B_{r_k}) \to 0$;

\item $\bar T_k$ locally converge (and in the Hausdorff sense for what
concerns the supports) to an $m$-plane with multiplicity $Q$;

\item $\bar \cM_k$ locally converge to
the flat $m$-plane (without loss of generality $\pi_0$);

\item recalling \eqref{e:seq2bis},
\begin{align}\label{e:sing grande}
\cH^{m-2+\alpha}_\infty (\rD_Q (\bar{T}_k) \cap \B_1)
&\geq \eta'>0\, ,
\end{align}
for some positive constant $\eta'$.
\end{enumerate}
We can then consider the following definition for the 
blow-up maps
\[
N^b_k : B_{3}\subset\R^m \to \Iq (\R^{m+n})
\]
given by
\begin{equation}\label{e:sospirata_successione}
N^b_k (x) := \bh_k^{-1} \bar{N}_k (\be_k (x))\, , \quad \text{with }
\bh_k:=\|\bar N_k\|_{L^2(\cB_{\frac32})},
\end{equation}
where $\be_k: B_3\subset\R^m\simeq T_{\bar{p}_k} \bar\cM_k\to \bar\cM_k$ denotes the exponential map
at $\bar{p}_k = t_k\,\Phii_{k} (0)/\bar{r}_k$.

Proposition~\ref{c:rev_Sob} implies then 
that there exists a constant $C>0$ such that, for every $k$,
\begin{equation}\label{e:rev_Sob2}
\int_{B_{\frac32}} |D{N}^b_k|^2 \leq C. 
\end{equation}
Moreover, as a simple consequence of Theorem~\ref{t:approxN} (details left to the readers),
we find an exponent $\gamma>0$ such that
\begin{gather}
\Lip (\bar{N}_k) \leq C \bh_k^\gamma,\label{e:Lip_riscalato}\\
\mass((\bT_{\bar{F}_k} - \bar{T}_k) \res (\p_k^{-1} (\cB_\frac32)) \leq C \bh_k^{2+2\gamma},\label{e:errori_massa_1000}\\
\int_{\cB_\frac32} |\etaa\circ \bar{N}_k| \leq C \bh_k^2\label{e:controllo_media}\, .
\end{gather}
It then follows from \eqref{e:rev_Sob2}, $\|N_k^b\|_{L^2(\cB_{3/3})} \equiv 1$ and the
Sobolev embedding for $Q$-valued functions (cp.~\cite[Proposition~2.11]{DS1}) that
up to subsequences (as usual not relabeled) there exists a Sobolev function
$N_\infty^b: B_{\frac32} \to \Iq (\R^{m+n})$
such that the maps $N^b_k$ converge strongly in $L^2(B_\frac32)$ to
$N^b_\infty$.
Then from \eqref{e:controllo_media} we deduce also that
\begin{gather}
\etaa \circ N^b_\infty \equiv 0
\quad\text{and}\quad
\|N_k^b\|_{L^2(\cB_{3/3})} \equiv 1. 
\end{gather}
Moreover, since the $\bar N_k$ are $\bar \cM_k$-normal approximations and
the $\bar\cM_k$ converging to the flat $m$-dimensional plane $\R^m \times\{0\}$,
$N^b_\infty$ takes values
in the space of $Q$-points of $\{0\}\times \R^n$ (in place of the full
$\R^{m+n}$).

To conclude our contradiction argument, we need to prove
the $N^b_\infty$ is Dir-minimizing.

\subsubsection{$N_\infty^b$ is Dir-minimizing}
Apart from the necessary technicalities, the proof of this claim is very intuitive and
relies on the following observation: if the energy of $N_\infty^b$ could be decreased,
then one would be able to find a rectifiable current with less mass then $\bar T_k$, because
the rescaling of $N_k^b$ are done in terms of the $L^2$ norm $\bh_k$ whereas
the errors in the normal approximation are superlinear with $\bh_k$.

Next we give all the details for this arguments.

\medskip

We can consider for every $\bar\cM_k$ an orthonormal frame of $(T\bar\cM_k)^\perp$, 
\[
\nu^k_1, \ldots, \nu^k_{{n}},
\]
with the property (cf.~\cite[Lemma~A.1]{DS2}) that
\[
\nu^k_j \to e_{m+j}  \quad \mbox{in $C^{2,\kappa/2}(\bar \cM_k)$ as $k\uparrow\infty$ for every $j$}
\]
(here $e_1, \ldots , e_{m+n}$ is the standard basis of $\mathbb R^{m+n}$).

Given now any $Q$-valued map 
$u = \sum_i\a{u_i}: \bar{\cM}_k \to \Iq (\{0\}\times \R^{n})$,
we can consider the map
\[
\mathbf{u}_k: x\mapsto \sum_i \a{ (u_i(x))^j \nu_j^k (x)},
\]
where we set $(u_i)^j:= \langle u_i (x), e_{m+j}\rangle$  and we 
use Einstein's convention.
Then, the differential map
$D\mathbf{u}_k := \sum_i \a{D(\mathbf{u}_k)_i}$ is given by
\begin{align*}
D(\mathbf{u}_k)_i &= D(u_i)^j  \nu_j^k + (u_i)^j D \nu_j^k . 
\end{align*}
Taking into account that $\|D\nu_i^k\|_{C^0} \to 0$
as $k\to +\infty$, we deduce that
\begin{align}
\left|\int \big(|D \mathbf{u}_k|^2 - |D u|^2\big)\right| 
\leq  o(1) \int \left(|u|^2+|Du|^2\right)\, .\label{e:stimozza}
\end{align}
Note that $N_k^b$ has also the form $\mathbf{u}_k^b$ for some
$Q$-valued function $u_k^b : \bar{\cM}_k \to \Iq (\{0\}\times \R^{n})$.

\medskip

We now show the $\D$-minimizing property of $N^b_\infty$.
There is nothing to prove if its Dirichlet energy vanishes.
We can therefore assume that there exists $c_0>0$ such that
\begin{equation}\label{e:reverse_control}
c_0 \bh_k^2 \leq \int_{\cB_\frac32} |D\bar{N}_k|^2\, .
\end{equation}
We argue by contradiction and assume there is a radius $t \in \left]\frac54,\frac32\right[$ and a 
function $f: B_\frac32 \to \Iq(\{0\} \times \R^{ n})$ such that
\[
f\vert_{B_\frac32\setminus B_t} = N^b_\infty\vert_{B_\frac32 \setminus B_t} \quad\text{and}\quad
\D(f, B_t) \leq \D(N^b_\infty, B_t) -2\,\delta,
\]
for some $\delta>0$.

Using $f$ as a model, we need to find a sequence of functions
$v_k^b$ such that they have the same boundary data of $N_k^b$ and less energy.
This can be done because of the strong convergence of the traces and the
possibility to make an interpolation between two functions with
close by traces.
This is one of the instances where thinking to multiple valued functions
as classical single valued ones may be useful.
In any case, the details are given in \cite[Proposition~3.5]{DS3} and 
lead to competitor functions $v^b_k$ such that,
for $k$ large enough,
\begin{gather*}
v^b_k\vert_{\de B_r} = N^b_k\vert_{\de B_r}, \quad
\Lip (v^b_k) \leq C \bh_k^\gamma,
\\
\int_{B_\frac32} |\etaa\circ v^b_k| \leq C \bh_k^2\quad\text{and}\quad
\int_{B_\frac32} |Dv^b_k|^2 \leq \int |DN^b_k|^2 - \delta \, \bh_k^2\, ,
\end{gather*}
where $C>0$ is a constant independent of $k$. 
Clearly, setting $\tilde{N}_k = v^b_k  \be_k^{-1}$ satisfy 
\begin{gather*}
\tilde{N}_k \equiv \bar{N}_k \quad \mbox{in $\cB_\frac32\setminus\cB_t$},\quad
\Lip (\tilde{N}_k) \leq C \bh_k^\gamma,\\
\int_{\cB_\frac32} |\etaa\circ \tilde{N}_k| \leq C \bh_k^2
\quad\text{and}\quad
\int_{\cB_\frac32} |D\tilde{N}_k|^2 \leq \int_{\cB_\frac32} |D\bar{N}_k|^2 - \delta \bh_k^2.
\end{gather*}

\medskip

Consider finally the map
$\tilde{F}_k (x) = \sum_i \llbracket x+\tilde{N}_i (x)\rrbracket$. The current $\bT_{\tilde{F}_k}$ coincides with 
$\bT_{\bar{F}_k}$ on $\p_k^{-1}( \cB_\frac32\setminus \cB_t)$.
Define the function $\varphi_k(p) = \dist_{\bar{\cM}_k} (0, \p_k (p))$
and consider for each $s\in \left]t,\frac32\right[$ the slices 
$\langle \bT_{\tilde{F}_k} - \bar{T}_k, \varphi_k, s\rangle$.
By \eqref{e:errori_massa_1000} we have
\[
\int_t^\frac32 \mass (\langle \bT_{\tilde{F}_k} - \bar{T}_k, \varphi_k, s\rangle) \leq C \bh_k^{2+\gamma}\, .
\]
Thus we can find for each $k$ a radius $\sigma_k\in \left]t, \frac32\right[$ on which 
$\mass (\langle \bT_{\tilde{F}_k} - \bar{T}_k, \varphi_k, \sigma_k\rangle) \leq C \bh_k^{2+\gamma}$. By the isoperimetric
inequality (see \cite[Remark 4.3]{DS3}) there is a current $S_k$ such that
\[
\partial S_k = 
\langle \bT_{\tilde{F}_k} - \bar{T}_k, \varphi_k, \sigma_k\rangle,
\quad \mass (S_k) \leq C \bh_k^{(2+\gamma)m/(m-1)}.
\]
Our competitor current is, then, given by
\[
Z_k := \bar{T}_k\res (\p_k^{-1} (\bar\cM_k\setminus \cB_{\sigma_k})) + S_k + \bT_{\tilde{F}_k}\res
(\p_k^{-1} (\cB_{\sigma_k})).
\]
Note that $Z_k$ has the same boundary as 
$\bar{T}_k$. On the other hand, by \eqref{e:errori_massa_1000} and the bound on $\mass (S_k)$, we have
\begin{equation}\label{e:perdita_massa}
\mass (\tilde{T}_k) - \mass (\bar{T}_k) \leq \mass (\bT_{\bar F_k}) - \mass (\bT_{\tilde{F}_k}) + C \bh_k^{2+2\gamma}\, .
\end{equation}
Denote by $A_k$ and by $H_k$ respectively the second fundamental forms and mean curvatures of the
manifolds $\bar\cM_k$. Using the Taylor expansion
of \cite[Theorem~3.2]{DS2}, we achieve
\begin{align}
 \mass (\tilde{T}_k) - \mass (\bar{T}_k) &\leq \frac{1}{2} \int_{\cB_\rho} \left(|D\tilde{N}_k|^2 - |D\bar{N}_k|^2\right)\notag\\ 
&\qquad+ C \|H_k\|_{C^0} \int \left(|\etaa\circ \bar{N}_k| + |\etaa\circ \tilde{N}_k|\right)\nonumber\\
&\qquad + \|A_k\|_{C^0}^2 \int \left( |\bar{N}_k|^2 + |\tilde{N}_k|^2\right) + o (\bh_k^2)\notag\\
&\leq -\frac{\delta}{2} \bh_k^2 + o (\bh_k^2)\, .\label{e:fine!!}
\end{align}
Clearly, \eqref{e:fine!!} and \eqref{e:perdita_massa}
contradict the minimizing property of $\bar{T}_k$ for $k$ large enough and 
this concludes the proof.

\subsection{Persistence of singularities}
We discuss step (G) of \S~\ref{ss:sketch}:
we show that the assumptions \eqref{e:seq2bis} and \eqref{e:seq3bis}
contradict Theorem~\ref{t:Qvalued}, which asserts that the singular
set of $N_\infty^b$ has $\cH^{m-2+\alpha}$ measure zero.

\medskip

Set
\[
\Upsilon := \left\{x\in \bar B_{1} : N^b_\infty(x) = Q\a{0}\right\},
\]
and note that, since $\etaa \circ N^b_\infty\equiv 0$ and
$\| N^b_\infty\|_{L^2(B_{\frac{3}{2}})}= 1$, from Theorem~\ref{t:Qvalued}
it follows that $\cH^{m-2+\alpha}_\infty (\Upsilon) = 0$.

The main line of the contradiction argument can be summarized in three steps.

\begin{enumerate}
\item By \eqref{e:seq2bis} and \eqref{e:seq3bis} there exists a set $\Lambda_k \subset \D_Q(\bar T_k)$
such that 
\[
\dist(\Lambda_k, \Upsilon) > c_1>0 \quad \text{and}\quad
\cH^{m-2+\alpha}_\infty(\Lambda_k)>c_2>0,
\]
for suitable constants $c_1,c_2>0$.

The key aspect of the set $\Lambda_k$ is the following:
by the H\"older regularity of Dir-minimizing functions
in Theorem~\ref{t:Qvalued}, the normal approximation $\bar N_k$ must
be big in modulus around any point in $\Lambda_k$.

\item Moreover, it follows from the Lipschitz approximation Theorem~\ref{t:approx} (see Theorem~\ref{t:persistence} below
that around any multiplicity $Q$ point 
of the current the energy of the Lipschitz approximation
is large enough with respect to the $L^2$ norm (cp.~\cite[Theorem~1.7]{DS3}).
This is what we call \textit{persistence of $Q$-point} phenomenon, and is in fact the
analytic core of this part of the proof.

We moreover stress that this part of the proof (even if 
it is not apparent from our exposition)
also uses the \textit{splitting-before-tilting}
estimates.

\item Putting together the previous two steps, we then conclude that
there is a big part of the current where the energy of the Lipschitz approximation
is large enough: matching the constant in the previous estimates, one realizes that
this cannot happen on a set of positive $\cH^{m-2+\alpha}$ measure.
\end{enumerate}

As usual, the actual proof is much more involved of the heuristic scheme above.
In the following we try to give some more explanations, referring
to \cite{DS3,DS4,DS5} for the detailed proof.

\medskip

\textit{Step (1).} 
We cover $\Upsilon$ by balls $\{\B_{\sigma_i} (x_i)\}$ in such a way that
\[
\sum_i \omega_{m-2+\alpha} (4\sigma_i)^{m-2+\alpha} \leq 
\frac{\eta'}{2},
\]
where $\eta'>0$ is the constant in \eqref{e:sing grande}.
By the compactness of $\Upsilon$, such a covering can be chosen finite.
Let $\sigma>0$ be a radius whose specific choice will be given only at the very end,
and such that $0 < 40 \, \sigma \leq \min \sigma_i$.
Denote by $\Lambda_{k}$ the set of $Q$ points of $\bar T_k$ far away from the singular set
$\Upsilon$:
\[
\Lambda_{k} := \big\{ p \in D_Q(\bar T_k) \cap \B_1 : \dist(p, \Upsilon)>4\min \sigma_i \big\}.
\]
Clearly, $\cH^{m-2+\alpha}_\infty (\Lambda_k)
\geq \frac{\eta'}{2}$. Let $\mathbf{V}$ denote the neighborhood of $\Upsilon$ of size $2\min \sigma_i$.
By the H\"older continuity of $\D$-minimizing
functions in Theorem~\ref{t:Qvalued} (ii), there is a positive constant $\vartheta>0$ such that
$|N^b_\infty (x)|^2\geq 2 \vartheta$ for every $x\not\in \mathbf{V}$. 
It then follows that
\[
 2 \,\vartheta \leq \mint_{B_{2\sigma} (x)} |N^b_\infty|^2  \qquad \mbox{$\forall\;x\in B_{\frac54}$
with  $\dist (x, \Upsilon) \geq 3\min \sigma_i$,}
\]
and therefore, for sufficiently large $k$'s,
\begin{equation}\label{e:dall'alto}
\vartheta \, \bh_k^2 \leq \mint_{\cB_{2\sigma} (x)} \cG (\bar{N}_k, Q \a{\etaa\circ \bar{N}_k})^2,
\end{equation}
for all $x\in \Gamma_k:= \p_{\bar\cM_k} (\Lambda_k)$.
This is the claimed lower bound on the modulus of $\bar N_k$.

\medskip

\textit{Step (2).}
This is the most important step of the proof.
We start introducing the following notation.
For every $p\in \Lambda_k$, consider $\bar{z}_k (p) = \p_{\pi_0} (p)$ and 
$\bar{x}_k (p) := \bar \Phii_k \in \bar\cM_k$, where $\bar \Phii_k$ is the induced
parametrization.

The key claim is the following:
there exists a geometric constant $c_0>0$ (in particular, 
independent of $\sigma$) such that, when $k$ is large enough,
for each $p\in \Lambda_k$ there is a radius $\varrho_p \leq 2\sigma$ with the following properties:
\begin{gather}
\frac{c_0\, \vartheta}{\sigma^{\alpha}}  \bh_k^2\leq \frac1{\varrho_p^{m-2+\alpha}}
\int_{\cB_{\varrho_p} (\bar{x}_k (p))} |D\bar{N}_k|^2,
\label{e:dal_basso_finale}\\
\cB_{\varrho_p} (\bar{x}_k (p)) \subset \B_{4\varrho_p} (p)\label{e:inside}\, .
\end{gather}

We show here the main heuristics leading to \eqref{e:dal_basso_finale}
(and we warn the reader that these are not the complete arguments), referring to \cite{DS5}
for \eqref{e:inside}.
The key estimate in this regard is the following:
there exists a constant $\bar s <1$ such that
\begin{equation*}\label{e:sotto_dal_cm}
\mint_{\cB_{\bar{s} \ell (L_k)} (x_k)} \cG (N_{j(k)}, Q \a{\etaa\circ N_{j(k)}})^2 
\leq \frac{\vartheta}{4 \omega_m \ell (L_k)^{m-2}} \int_{\cB_{\ell (L_k)} (x_k)} |DN_{j(k)}|^2\,,
\end{equation*}
that is, rescaling to $\bar\cM_k$,
there exists $t (p) \leq \bar \ell_k$ such that
\begin{equation}\label{e:sotto_dal_cm_2}
\mint_{\cB_{\bar{s} t (p)} (\bar{x}_k) (p)} \cG (\bar{N}_k, Q \a{\etaa\circ \bar{N}_k})^2 \leq 
\frac{\vartheta}{4 \omega_m t (p)^{m-2}} \int_{\cB_{t (p)} (\bar{x}_k (p))} |D\bar{N}_k|^2\, .
\end{equation}
We show that we can choose $\varrho_p \in ]\bar s\, t (p), 2\sigma[$
such that \eqref{e:dal_basso_finale} follows from \eqref{e:sotto_dal_cm_2}.
To this aim we can distinguish two cases.
Either
\begin{equation}\label{e:alto_rho_x}
\frac{1}{\omega_m t (p)^{m-2}} \int_{\cB_{t (p)} (\bar{x}_k (p))} |DN_k|^2 \geq \bh_k^2\,,
\end{equation}
and \eqref{e:dal_basso_finale} follows with $\varrho_p = t (p)$.
Or \eqref{e:alto_rho_x} does not hold, and we argue as follows.
We use first
\eqref{e:sotto_dal_cm_2} to get
\begin{equation}\label{e:L2_basso}
 \mint_{\cB_{\bar{s} t (p)} (\bar{x}_k (p))} \cG (\bar{N}_k, Q \a{\etaa\circ \bar{N}_k})^2
\leq \frac{\vartheta}{4} \bh_k^2\, . 
\end{equation}
Then, we show by contradiction that there exists a radius $\varrho_y \in [\bar s t (p), 2\sigma]$
such that \eqref{e:dal_basso_finale} holds. Indeed, if this were not the case, 
setting for simplicity $f:= \cG (\bar{N}_k, Q\a{\etaa\circ \bar{N}_k})$
and letting $j$ be the smallest integer such that
$2^{-j} \sigma \leq \bar{s} t (p)$, we can estimate as follows
\begin{align}
\mint_{\cB_{2\sigma} (\bar{x}_k (p))} f^2
\leq\;& 2 \mint_{\cB_{\bar{s} t (p)} (\bar{x}_k (p))} 
f^2
+ \sum_{i=0}^{j} \left(\mint_{\cB_{2^{1-i}\sigma} (\bar{x}_k (p))} f^2 - \mint_{\cB_{2^{-i}\sigma} (\bar{x}_k (p))} f^2\right)\nonumber\\
\stackrel{\mathclap{\eqref{e:L2_basso}}}{\leq}
\; & \; \frac{\vartheta}{2} \bh_k^2 + C \sum_{i=1}^{j} \frac{1}{(2^{-j}\sigma)^{m-2}}
\int_{\cB_{2^{1-i}\sigma} (\bar{x}_k (p))} |D\bar{N}_k|^2\nonumber\\
\leq\; & \frac{\vartheta}{2} \bh_k^2 + C c_0\frac{\vartheta}{\sigma^\alpha} 
\bh_k^2 \sum_{i=1}^{j} (2^{-j}\sigma)^\alpha
\leq \bh_k^2 \left(\frac{\vartheta}{2} + C(\alpha) c_0 \vartheta\right)\, .\nonumber
\end{align}
In the second line we have used the simple Morrey inequality 
\begin{align*}
\left|\mint_{\cB_{2t} (\bar{x}_k (p))} f^2 - \mint_{\cB_{t} (\bar{x}_k (p))} f^2\right| &\leq \frac{C}{t^{m-2}} \int_{\cB_{2t} (\bar{x}_k (p))}
|Df|^2\\
&\leq  \frac{C}{t^{m-2}} \int_{\cB_{2t} (\bar{x}_k (p))} |D\bar{N}_k|^2\, .
\end{align*}
The constant $C$ depends only upon the 
regularity of the underlying manifold $\bar\cM_k$, 
and, hence, can assumed independent of $k$.

Since $C(\alpha)$ depends only on $\alpha$, $m$ and $Q$,
for $c_0$ chosen sufficiently small the latter inequality would contradict \eqref{e:dall'alto}.

\medskip

\textit{Step (3).}
We collect the estimates \eqref{e:dal_basso_finale} and \eqref{e:inside} to
infer the desired contradiction.
We cover $\Lambda_k$ with balls $\B^i := \B_{20 \varrho_{p_i}} (p_i)$
such that $\B_{4\varrho_{p_i}} (p_i)$ are disjoint, and deduce
\begin{align*}
\frac{\eta'}{2} & \leq C(m)
\sum_i \varrho_{p_i}^{m-2+\alpha}
\stackrel{\eqref{e:dal_basso_finale}}{\leq} \frac{C(m)}{c_0}\frac{\sigma^\alpha}{\vartheta\bh_k^2} \sum_i \int_{\cB_{\varrho_{p_i}} (\bar{x}_k (p_i))} |D\bar{N}_k|^2\\
&\leq \frac{C(m)}{c_0}\frac{\sigma^\alpha}{\vartheta\bh_k^2} \int_{\cB_{\frac32}} |D\bar{N}_k|^2
\stackrel{\eqref{e:rev_Sob2}}{\leq}C\frac{\sigma^\alpha}{\vartheta},
\end{align*}
where $C(m)>0$ is a dimensional constant. We have used that
the balls $\cB_{\varrho_{p_i}} (\p_{\bar\cM_k}(p_i))$ are pairwise disjoint
by \eqref{e:inside}.
Now note that $\vartheta$ and $c_0$ are independent of $\sigma$, and therefore we
can finally choose $\sigma$ small enough to lead to a contradiction.

\subsubsection{Persistence of $Q$-points}
Here we explain a simple instance of estimate \eqref{e:sotto_dal_cm_2},
reporting the following theorem from \cite{DS3}.

\begin{theorem}[Persistence of $Q$-points]\label{t:persistence}
For every $\hat{\delta}>0$, there is $\bar{s}\in ]0, \frac{1}{2}[$ such that, for every $s<\bar{s}$, there exists $\hat{\eps} (s,\hat{\delta})>0$ with the following property. If $T$ is as in Theorem \ref{t:approx},
$E := \bE(T,\bC_{4\,r} (x)) < \hat{\eps}$
and $\Theta (T, (p,q)) = Q$ at some $(p,q)\in \bC_{r/2} (x)$, then the approximation
$f$ of Theorem~\ref{t:approx} satisfies
\begin{equation}\label{e:persistence}
\int_{B_{sr} (p)} \cG (f, Q \a{\etaa\circ f})^2 \leq \hat{\delta} s^m r^{2+m} E\, .
\end{equation}
\end{theorem}

This theorem states that, in the presence of multiplicity $Q$ points of the
current, the Lipschitz (and therefore also the normal) approximations
must have a relatively small $L^2$ norm, compared to the excess; or,
as explained above, if in the normal approximation the excess is linked to the
Dirichlet energy (for example this is the case of (EX)-cubes in the Whitney decomposition),
the energy needs to be relatively large with respect to the $L^2$ norm, thus vaguely explaining
the link to \eqref{e:sotto_dal_cm_2}.

\begin{proof}
By scaling and translating we assume $x=0$ and $r=1$;
the choice of $\bar{s}$ will be specified at the very end, but for the moment we impose $\bar{s}<\textstyle{\frac{1}{4}}$. 
Assume by contradiction that, for arbitrarily small $\hat{\eps}>0$, there are currents $T$ 
and points $(p,q)\in \bC_{1/2}$ satisfying:
$E := \bE (T, \bC_4)< \hat{\eps}$, $\Theta (T, (p,q)) = Q$ and, for $f$ as in Theorem~\ref{t:approx},
\begin{equation}\label{e:no_persistence}
\int_{B_s (p)} \cG (f, Q \a{\etaa\circ f})^2 > \hat{\delta} s^m E\, .
\end{equation}
Set $\bar{\delta} = \frac{1}{4}$ and fix $\bar{\eta}>0$ (whose choice will be specified later).
For a suitably small $\hat{\eps}$ we can apply Theorem \ref{t:harmonic_final},
obtaining a Dir-minimizing approximation $w$.
If $\bar{\eta}$ and $\hat{\eps}$ are suitably small, we have
\[
\int_{B_s (p)} \cG (w, Q\a{\etaa\circ w})^2 \geq \textstyle{\frac{3\hat\delta}{4}} s^m E\, ,
\]
and $\sup\big\{\D (f), \D (w)\} \leq  C E$. Then there exists 
$\bar{p}\in B_s (p)$ with
\[
\cG (w (\bar p), Q \a{\etaa \circ w (\bar p)})^2 \geq \frac{3\hat\delta}{4\omega_m}\,E,
\]
and, by the H\"older continuity in Theorem~\ref{t:Qvalued} (ii), we conclude
\begin{align}\label{e:imponi}
g (x) & := \cG (w (x), Q\a{\etaa\circ w (x)}) \notag\\
&\geq \left(\textstyle{\frac{3\hat\delta}{4\omega_m}} E\right)^{\frac{1}{2}} - 
2\,(C E)^{\frac{1}{2}} \bar{C} \bar{s}^{\kappa}
\geq \left(\textstyle{\frac{\hat\delta}{2}} E\right)^{\frac{1}{2}}\,,
\end{align}
where we assume that $\bar s$ is chosen small enough in order to satisfy the last inequality.
Setting $h (x):= \cG (f (x), Q\a{\etaa\circ f (x)})$, we recall that we have
\[
\int_{B_s (p)} |h-g|^2 \leq C\, \bar{\eta} E\, .
\]
Consider therefore the set $A:= \big\{h > \big(\frac{\hat \delta}{4} E\big)^{\frac{1}{2}}\big \}$. If $\bar \eta$ is sufficiently small, we can assume that 
\[
|B_s(p)\setminus A| < \frac{1}{8} |B_s|.
\]
Further, define $\bar{A}:= A\cap K$, where $K$ is the set of Theorem \ref{t:approx}. Assuming $\hat{\eps}$ is sufficiently small we ensure $|B_s(p)\setminus \bar A| < \frac{1}{4} |B_s|$. Let $N$ be the smallest integer such that $N \frac{\hat \delta E}{64 Q s} \geq \frac{s}{2}$. Set
\[
\sigma_i := s- i \frac{\hat \delta E}{64Q s} \quad\text{for $i\in \{0, 1\ldots, N\}$,}
\]
and consider, for $i\leq N-1$,  the annuli $\cC_i:= B_{\sigma_i} (p) \setminus B_{\sigma_{i+1}} (p)$. If $\hat{\eps}$ is sufficiently small, we can assume that $N\geq 2$ and
$\sigma_N \geq \frac{s}{4}$.
For at least one of these annuli we must have $|\bar A\cap \cC_i|\geq \frac{1}{2} |\cC_i|$. We then let $\sigma:= \sigma_i$ be the corresponding outer radius and we denote by $\cC$ the corresponding annulus.

Consider now a point $x\in \cC \cap \bar{A}$ and let $T_x$ be the slice $\langle T, \p, x \rangle$. Since $\bar{A}\subset K$, for a.e.~$x\in \bar{A}$ we have $T_x = \sum_{i=1}^Q \a{(x,f_i (x))}$. Moreover, there exist $i$ and $j$ 
such that $|f_i (x)-f_j (x)|^2\geq \frac{1}{Q} \cG (f (x), \a{\etaa \circ f (x)})^2 \geq \frac{\hat \delta}{4Q} E$ (recall
that $x\in \bar{A}\subset A$). When $x\in \cC$ and the points $(x,y)$ and $(x,z)$ belong both to $\B_\sigma ((p,q))$, we must have 
\[
|y-z|^2 \leq 4 \left(\sigma^2 - \Big(\sigma - \textstyle{\frac{\hat{\delta} E}{64 Q s}}\right)^2\Big) \leq 
\textstyle{\frac{\sigma \hat\delta E}{8Q s}} \leq \textstyle{\frac{\hat\delta E}{8Q}}\, .
\]
Thus, for $x\in \bar{A}\cap \cC$ at least one of the points $(x, f_i (x))$ is not contained in $\B_\sigma ((p,q))$. We conclude therefore 
\begin{align}
\|T\| (\bC_{\sigma} (p) \setminus \B_\sigma ((p,q))) &\geq |\cC\cap \bar{A}| \geq \frac{1}{2} |\cC| = 
\frac{\omega_m}{2} \left( \sigma^m - \left(\sigma -  \textstyle{\frac{\hat{\delta} E}{64 Qs}}\right)^m \right)\nonumber\\
&\geq \frac{\omega_m}{2} \sigma^m \left( 1- \left(1-  \textstyle{\frac{\hat{\delta} E}{64 Qs \sigma}}\right)^m \right)\, .
\end{align}
Recall that, for $\tau$ sufficiently small, $(1-\tau)^m \leq 1 - \frac{m\tau}{2}$.
Since $\sigma \geq \frac{s}{4}$, if $\hat{\eps}$ is chosen sufficiently small we can therefore conclude
\begin{equation}\label{e:punti_persi}
\|T\| (\bC_{\sigma} (p) \setminus \B_\sigma (p)) \geq \frac{\omega_m \sigma^m \hat{\delta} E}{256 Q s \sigma}
\geq \frac{\omega_m}{1024 Q} \hat{\delta} E  \sigma^{m-2}= c_0 \hat{\delta} E \sigma^{m-2}\, .
\end{equation}
Next, by Theorem \ref{t:approx} and Theorem \ref{t:harmonic_final},
\begin{equation}\label{e:stima_nel_cilindro}
\|T\| (\bC_{\sigma} (p)) \leq Q \omega_m \sigma^m + C E^{1+\gamma_1} +\bar{\eta} E + \int_{B_\sigma (p)} \frac{|Dw|^2}{2}\, .
\end{equation}
Moreover, as shown in \cite[Proposition 3.10]{DS1}, we have
\begin{equation}\label{e:decay}
\int_{B_\sigma (p)} |Dw|^2 \leq  C  \D (w) \sigma^{m-2+2\kappa},
\end{equation}
(for some constants $\kappa$ and $C$ depending only on $m$, $n$ and $Q$; in fact the exponent
$\kappa$ is the one of Theorem~\ref{t:Qvalued} (ii)).
Combining \eqref{e:punti_persi},
\eqref{e:stima_nel_cilindro} and \eqref{e:decay}, we conclude
\begin{equation}\label{e:da_sopra}
\|T\| (\B_\sigma ((p,q))) \leq Q \omega_m \sigma^m + \bar{\eta}\, E + C E^{1+\gamma_1} + C E \sigma^{m-2+2\kappa} - c_0 \sigma^{m-2} \hat\delta E\, .
\end{equation}
Next, by the monotonicity formula, $\rho\mapsto \rho^{-m} \|T\| (\B_\rho ((p,q)))$ is a monotone function.
Using $\Theta(T,(p,q)) =Q$, we conclude
\begin{equation}\label{e:da_sotto}
\|T\| (\B_\sigma ((p,q))) \geq Q \omega_m \sigma^m.
\end{equation}
Combining \eqref{e:da_sopra} and \eqref{e:da_sotto} we conclude 
\begin{equation}\label{e:basso_alto_10}
C \sigma^2 + (\bar{\eta} + C E^\gamma_1) \sigma^{2-m} + C \sigma^{2\kappa} \geq c_0 \hat{\delta}\, .
\end{equation}
Recalling that $\sigma \leq s < \bar{s}$, we can, finally, specify $\bar{s}$: it is chosen so that
$C \bar{s}^2 + C \bar{s}^{2\kappa}$ is smaller than $\frac{c_0}{2} \hat{\delta}$. Combined with
\eqref{e:imponi} this choice of $\bar{s}$ depends, therefore, only upon $\hat{\delta}$. \eqref{e:basso_alto_10}
becomes then
\begin{equation}\label{e:basso_alto_11}
(\bar{\eta} + C E^{\gamma_1}) \sigma^{2-m}  \geq \textstyle{\frac{c_0}{2}} \hat{\delta}\, .
\end{equation}
Next, recall that $\sigma\geq \frac{s}{4}$. We then choose $\hat{\eps}$ and $\bar \eta$
so that
$(\bar{\eta} + C \hat{\eps}^{\gamma_1}) (\frac{s}{4})^{2-m} \leq \frac{c_0}{4} \hat{\delta}$.
This choice is incompatible with \eqref{e:basso_alto_11}, thereby reaching a contradiction:
for this choice of the parameter $\hat{\eps}$ (which in fact depends only upon $\hat{\delta}$
and $s$) the conclusion of the theorem, i.e.~\eqref{e:persistence}, must then be valid.
\end{proof}

\section{Open questions}
We close this survey recalling some open problems concerning the
regularity of area minimizing integer rectifiable currents.
Some of them have been only slightly touched and would actually explain some of the
complications that we met along the proof 
of the partial regularity result.

For more open problems and comments, we suggest the reading of \cite{proceedings84,DLicm}.

\subsubsection*{(A)}
One of the main, perhaps the most well-known, open problems
is the uniqueness of the tangent
cones to an area minimizing current, i.e.~the uniqueness 
of the limit $(\iota_{x,r})_\sharp T$ as $r\to 0$ for every $x \in \supp(T)$.
The uniqueness is known for two dimensional currents (cp.~\cite{Wh83}),
and there are only partial results in the general case (see \cite{AA,Sim83}).

We have run into this issue in dealing with the step (C) of \S~\ref{ss:sketch}, 
because it is one of the possible reasons why a center manifold may be
sufficient in our proof.

\subsubsection*{(B)}
A related question is that of the uniqueness of the inhomogeneous
blowup for Dir-minimizing $Q$-valued functions.
Also in this case the uniqueness is known for two dimensional domains 
(cp.~\cite{DS1}, following ideas of \cite{Chang}).

Even if it does not play a role in the contradiction argument
for the partial regularity,
a positive answer to this question could indeed contribute to the solution
of next two other major open problems.

\subsubsection*{(C)} It is unknown whether the singular set of an area
minimizing current has always locally finite $\cH^{m-2}$ measure.
This is the case for two dimensional currents (as proven by Chang \cite{Chang}); note
that in this result the uniqueness of the blowup Dir-minimizing map
plays a fundamental role.

\subsubsection*{(D)} It is unknown whether the singular set of an area
minimizing current has some geometric structure, e.g.~if it is rectifiable
(i.e., roughly speaking, if it is contained in lower dimensional ($m-2$)-dimensional submanifolds).
Once again it is known the positive answer for two dimensional currents,
where the singularities are known to be locally isolated,
and the uniqueness
of the tangent map is one of the fundamental steps in the proof.

\subsubsection*{(E)} We mention also the problem of finding
more example of area minimizing currents, other than those coming from
complex varieties or similar calibrations.
Indeed, our understanding of the possible pathological behaviors of such
currents is pretty much limited by the few examples we have at disposal.
In particular, it would be extremely interesting
to understand if there could be minimizing currents
with weird singular set (e.g., of Cantor type).

\subsubsection*{(F)} Finally, we mention the problem of boundary regularity
for higher codimension area minimizing currents, which to our knowledge is
mostly open.

%
%

\bibliographystyle{plain}
\bibliography{references}

\end{document}